\documentclass{paper}
\usepackage[T1]{fontenc}
\usepackage[utf8]{inputenc}

\usepackage{amsmath,amsfonts,amssymb,amsthm,bbold,bbm,mathtools,graphicx}
\usepackage{pgf}
\usepackage{comment}
\usepackage{tikz-cd}
\usetikzlibrary{hobby}

\usetikzlibrary{arrows}
\usetikzlibrary{cd}

\newtheorem{theorem}{Theorem}[section]
\newtheorem{proposition}{Proposition}[section]
\newtheorem{corollary}{Corollary}[section]
\newtheorem{remark}{Remark}[section]
\newtheorem{lemma}{Lemma}[section]

\newtheorem{assumption}[theorem]{Assumption}
\usepackage[a4paper,margin=0mm,left=30mm,right=30mm,top=25mm,bottom=25mm, headsep=10mm,footskip=15mm,headheight=30mm]{geometry}
\usepackage{hyperref}
\def\({\left(}
\def\){\right)}

\def\diag{\mathrm{diag}}
 
\def\to{\rightarrow}

\def\bar{\overline}

\def\<{\big<} \def\>{\big>}
\def\RR{\mathbbm R} \def\SS{\mathcal S}
\def\CC{\mathbbm C}

\renewcommand{\bar}{\overline}
\renewcommand{\tilde}{\widetilde}

\def\Delta{\triangle}
\def\top{T}

\def\saufzero{\setminus\{0\}}
\newcommand{\ps}[2]{\left\langle #1 , #2 \right\rangle}
\newcommand{\norm}[1]{\| #1 \|}
\newcommand{\metspec}{\mathbbm{M}}
\newcommand{\I}{\mathcal I}

\renewcommand{\epsilon}{\varepsilon}

\newcommand{\ones}{\mathbf{1}}

\newcommand{\new}[1]{{\color{magenta}#1}}

\usepackage[normalem]{ulem}

\usepackage{comment}

\begin{document}
\title{Computing the norm of nonnegative matrices and the log-Sobolev constant of Markov chains}

\author{Antoine Gautier\thanks{QuantPi, Saarbr\"ucken (Germany)}
\and Matthias Hein\thanks{Deptartment of Computer Science, University of T\"ubingen, T\"ubingen (Germany)}
\and Francesco Tudisco\thanks{School of Mathematics, GSSI Gran Sasso Science Institute, L'Aquila (Italy)}
}

\maketitle

\begin{abstract}
We analyze the global convergence of the power iterates for the computation of a general mixed-subordinate matrix norm. We prove a new global convergence theorem for a class of entrywise nonnegative matrices that generalizes %to any pair of differentiable and strongly monotonic norms the 
and improves a well-known results for mixed-subordinate  $\ell^p$ matrix norms. %known convergence properties of the power sequence. 
In particular, exploiting the Birkoff--Hopf contraction ratio of nonnegative matrices, we obtain novel and explicit global convergence guarantees for a range of matrix norms whose computation has been recently proven to be NP-hard in the general case, including the case of mixed-subordinate norms induced by the vector norms made by the sum of different $\ell^p$-norms of subsets of entries. Finally, we use the new results combined with   hypercontractive inequalities to prove a new lower bound on the logarithmic Sobolev constant of a Markov chain. 
\\[.5em]
%47J10  % In Oper. Theory: Nonlinear spectral theory, nonlinear eigenvalue problems 
\\[.5em]
\end{abstract}

\maketitle

\maketitle

\section{Introduction}
Let $A$ be an $m\times n$ matrix and consider the matrix norm
$$\|A\|_{\beta \to \alpha}=\max_{x \neq 0}\frac{\|Ax\|_\alpha}{\|x\|_\beta}\, ,$$
where $\|\cdot\|_\alpha$ and $\|\cdot\|_\beta$ are vector norms. 

Computing $\|A\|_{\beta\to\alpha}$ is a classical problem in computational mathematics, as norms of this kind arise naturally in many situations, such as approximation theory, estimation of matrix condition numbers and approximation of relative residuals \cite{higham2002accuracy}. However, attention around the problem of computing $\|A\|_{\beta \to \alpha}$ has been growing in recent years. In fact, for example, matrix norms of this type can be used in combinatorial optimization and sparse data recovery,  to approximate generalized Grothendieck and restricted isometry constants \cite{allen2016restricted,candes2008restricted,Friedland2019,khot2012grothendieck}, in scientific computing, to estimate the largest entries of large matrices \cite{highamRelton}, in data mining and learning theory, to minimize empirical risks or obtain robust nonnegative graph embeddings \cite{combettes2018regularized,zhang2012robust}, or in quantum information theory and the study of Khot's unique game conjecture where the computational complexity of evaluating $\norm{A}_{\beta\to\alpha}$ plays an important role \cite{Harrow}. Moreover, it was observed by Lim in \cite{lim2005singular} that the notion of tensor norm and tensor spectrum relates to $\|A\|_{\beta\to\alpha}$ in a very natural way and thus relevant advances on the problem of computing $\|A\|_{\beta\to\alpha}$ when $A$ is entrywise nonnegative and $\|\cdot\|_\alpha$, $\|\cdot\|_{\beta}$ are $\ell^p$ norms have been recently obtained as a consequence of a number of new  nonlinear Perron-Frobenius-type theorems for higher-order maps \cite{Friedland2013,GHT_PF,GHT_tensors,gautier2019contractivity}.

Closed form solutions and efficient algorithms are known for some special $\ell^p$ norms, as for instance the case where  $\|\cdot\|_\alpha=\|\cdot\|_\beta$ and they coincide with either the $\ell^1$, the $\ell^2$, or the $\ell^\infty$ norm,  or the case where $p\leq 1\leq q$ and $\|\cdot\|_\alpha$  and $\|\cdot\|_\beta$ are $\ell^p$ and $\ell^q$ (semi) norms, respectively  (c.f.\ \cite{Drakakis2009, Lewis10atop,rohn2000computing}). However, the computation of $\|A\|_{\beta \to \alpha}$ is generally NP-hard \cite{hendrickx2010matrix,Steinberg2005}.

The best known method for the computation of $\|A\|_{\beta\to\alpha}$ is the (nonlinear) power method, essentially introduced by Boyd in \cite{Boyd1974} and then further analyzed and extended for instance in \cite{Bhaskara2011, Friedland2013,Higham1992, Tao1984}.    
When the considered vector norms are $\ell^p$ norms, the power method can count on a very fundamental global convergence result which ensures  convergence to the matrix norm $\|A\|_{\beta\to\alpha}$ for a class of entry-wise nonnegative matrices $A$ and for a range of $\ell^p$ norms. We discuss in detail the method and its convergence in Section \ref{sec:NPM}. 

The convergence of the method is a consequence of an elegant fixed point argument that involves a nonlinear operator $\mathcal S_A$ and its Lipschitz contraction constant. However, the convergence analysis of this method has two main  uncovered points: On the one hand, all the work done so far addresses only the case of $\ell^p$ norms whereas almost nothing is known about the global convergence behavior of the power iterates for more general norms. On the other hand, even for the case of $\ell^p$ norms, known upper-bounds on the contraction constant of $\mathcal S_A$ are not sharp, especially for positive matrices. In this work we provide novel results that address and improve both these directions.

Consider for example the case where $\|\cdot \|_\alpha$ is defined as 
\begin{equation}\label{eq:example_norm}
	\|x\|_\alpha = \|(x_1, \dots, x_k)\|_{p_1} + \|(x_{k+1},\dots, x_{n})\|_{p_2}
\end{equation}
where $k$ is a positive integer not larger than the dimension of $x$ and  $\|\cdot\|_{p_i}$ are $\ell^p$ norms. Of course one can extend this idea by looking at any family of subsets of entries of $x$ and any set of $\ell^p$ norms, in order to generate arbitrarily new norms. Norms of this form are natural modifications of $\ell^p$ norms and are used for instance to define the generalized Grothendieck constants as in \cite{khot2012grothendieck} or in graph matching problems to build continuous relaxation of the set of matrix permutations \cite{duchenne2011tensor,nguyen2017efficient}. However, even for this case, extending the result of Boyd is not straightforward.

In this work we consider general pairs of  monotonic and differentiable vector norms and provide a thorough convergence analysis of the power method for the computation of the corresponding induced matrix norm $\|A\|_{\beta \to \alpha}$. 
Our result is based on a novel nonlinear Perron-Frobenius theorem for this kind of norms and ensures global convergence of the power method provided that the Birkhoff contraction ratio of the power iterator is smaller than one. 

When applied to the case $\|A\|_{q\to p}$ of $\ell^p$ norms, our result does not only imply the current convergence result, but actually significantly  improves the range of values of $p$ and $q$ for which global convergence can be ensured. This is particularly interesting from a complexity viewpoint. In fact, for example, although the computation of $\|A\|_{q\to p}$ is well known to be NP-hard for $p>q$, we show that for a non-trivial class of nonnegative matrices the power method converges to $\|A\|_{q\to p}$ in polynomial time even for $p$ sensibly larger than $q$. To our knowledge this is the first global optimality result for this problem that does not require the condition $p\leq q$. 

In the general case $\|A\|_{\beta \to \alpha}$, a main computational drawback of the power method is related with the computation of the  dual norm $\|\cdot\|_{\beta^*}$. In fact, if $\|\cdot\|_\beta$ is not an $\ell^p$ norm,  the corresponding dual norm may be challenging to compute \cite{friedland2016computational}. In practice, evaluating $\norm{\cdot}_{\alpha^*}$ from $\norm{\cdot}_{\alpha}$ can be done via convex optimization and Corollary 7 of \cite{friedland2016computational} proves that $\norm{\cdot}_{\alpha^*}$ can be evaluated in polynomial time (resp. is NP-hard) if and only if $\norm{\cdot}_{\alpha}$ can be evaluated in polynomial time (resp. is NP-hard). There are norms for which an explicit expression in terms of arithmetic operations for $\norm{\cdot}_{\alpha}$ is given by construction (resp. modelisation), but such an expression is not available for the dual $\norm{\cdot}_{\alpha^*}$. As we  discuss in Section \ref{dual_section}, examples of this type include for instance $\norm{x}_{\alpha}=(\norm{x}^2_{p}+\norm{x}^2_{q})^{1/2}$. A further main result of this work addresses this issue for the particular case of norms of the type \eqref{eq:example_norm}. For this family of norms we provide an explicit convergence bound and an explicit formula for the power iterator for the computation of the corresponding matrix norm $\|A\|_{\beta \to \alpha}$. To illustrate possible applications of the result, we list in Corollaries \ref{corrr1}--\ref{corrrend} relatively sophisticated and non-standard matrix norms together with an explicit condition for their computability.

We conclude with a discussion on the connection between our result and the log-Sobolev constant of Markov chains. This constant induces estimates on the rate of convergence of the Markov chain to the equilibrium and has important applications in the analysis of Markov Chain Monte Carlo algorithms \cite{carbone,diaconis1996,goel,jerrum}. It is well-known that the log-Sobolev constant is upper bounded by half the spectral gap. However, obtaining lower bounds on the constant is much more difficult \cite{lacoste}. The log-Sobolev constant is connected to matrix norms through the celebrated hypercontractive inequalities \cite{bakry,gross} which characterize the log-Sobolev constant in terms of the weighted $\ell^{p,q}$-norms of the continuous time Markov semigroup induced by the chain. Moreover, these inequalities require $p>q$, which is precisely the range of parameters for which no previous global optimal algorithm was  known.  By exploiting these connections we obtain a new  lower bound for the log-Sobolev constant of a Markov chain.

We organize the discussion as follows: In Section \ref{sec:NPM} we review the nonlinear power method and its main convergence properties. In Section \ref{sec:cone-background} we review relevant preliminary cone-theoretic results and notation.  
Then, in Section  \ref{sec:PF}, we propose a novel and detailed global convergence analysis of the method based on a Perron-Frobenius type result for the map $x\mapsto \|Ax\|_\alpha/\|x\|_\beta$, in the case of entry-wise nonnegative matrices and monotonic norms $\|\cdot\|_\alpha, \|\cdot\|_\beta$.  We derive new  conditions for the global convergence to $\|A\|_{\beta\to\alpha}$ that, in particular, help shedding new light on the NP-hardness of the problem, and we propose a new explicit bound on the linear convergence rate of the power iterates. In Section \ref{sec:sum_p_norms} we focus on the particular case of norms of the same form as \eqref{eq:example_norm}. We show how to practically implement the power method for this type of norms, we prove a specific convergence criterion that gives a-priori global convergence  guarantees and discuss the complexity of the method. Finally, an application of the new Perron-Frobenius result giving estimation on the log-Sobolev constant of finite Markov chains is discussed in Section \ref{sec:log-sobolev}. First, we compute the norm $\|A\|_{\beta\to\alpha}$ where $A$ is a stochastic matrix and $\norm{\cdot}_{\alpha}$, $\norm{\cdot}_{\beta}$ are weighted $\ell^p$ norms. Then, we use this information together with the hypercontractive characterization of the log-Sobolev constant to derive new lower bounds for~it.

\section{Boyd's nonlinear power method}\label{sec:NPM}
Let $\|\cdot\|_p$, $\|\cdot\|_q$ be the usual $\ell^p$ and $\ell^q$ vector norms and consider the induced matrix norm $\|A\|_{q\to p} = \max_{x\neq 0}\|Ax\|_p / \|x\|_q$. A well known explicit formula holds for the $\ell^1$ and $\ell^\infty$ matrix norms $\|A\|_{1\to 1}$, $\|A\|_{\infty\to\infty}$. However, while the mixed norm $\|A\|_{1\to\infty}$  equals $\max_{ij}|a_{ij}|$, the computation of $\|A\|_{\infty\to 1}$ is NP-hard \cite{rohn2000computing}. More generally, when $p$ is any rational number $p\neq 1,2$, computing the norm $\|A\|_{p\to p}$ is NP-hard for a general matrix $A$ \cite{hendrickx2010matrix}, and the same holds for any norm $\|A\|_{q\to p}$, for  $1\leq p< q \leq \infty$ \cite{Steinberg2005}.  
The best known  technique to compute $\|A\|_{q\to p}$ is a form of nonlinear power method that we review in what follows.

Consider the nonnegative function $f_{A}(x) = \|Ax\|_p/\|x\|_q$.
The norm $\|A\|_{q\to p}$ is the global maximum of $f_A$ 
by analyzing the optimality conditions of $f_A$, for differentiable $\ell^p$-norms $\|\cdot\|_p$ and $\|\cdot\|_q$, 
we note that 
\begin{equation*}%\label{eq:critical_point_equation}
 \nabla f_A(x)=0 \Longleftrightarrow A^T J_p (Ax)=f_A(x)J_q(x) \, ,
\end{equation*}
where, for $1<p<\infty$, we denote by $J_p(x)$  the gradient of the norm $\nabla \|x\|_p = J_p(x)=\|x\|_p^{1-p}\, \Phi_p(x)$, with $\Phi_p(x)$ entrywise defined as $\Phi_p(x)_i = |x_i|^{p-2}x_i$. Let $p^*$ be the dual exponent such that $1/p+1/p^* =1$. As $J_{p^*}(J_p(x))=x/\|x\|_p$ for all $x\neq 0$ and $J_{p}(\lambda \, x) = J_p(x)$ for any coefficient $\lambda>0$, we have that $\nabla f_A(x)=0$ if and only if $J_{q^*}(A^TJ_p(Ax))= x/\|x\|_q$. Thus, $x$ with $\|x\|_{q}=1$ is a critical point of $f_A(x)$ if and only if it is a fixed point of the map $J_{q^*}(A^T J_{p}(Ax))$. 
The associated fixed point iteration 
\begin{equation}\label{eq:power_sequence_boyd}
x_0  = x_0 /\|x_0 \|_{q}, \quad x_{k+1} = J_{q^*}(A^TJ_p(Ax_k))\qquad \text{ for } k=0,1,2,3,\dots
\end{equation}
defines what we call (nonlinear) power method for $\|A\|_{q\to p}$. 

 Although, in practice,  the method applied to $\|A\|_{p\to p}$ for $p=1,\infty$ often seems to converge to the global maximum (see e.g.\ \cite{Higham1990}), no guarantees exist for the general case.  For differentiable $\ell^p$ norms and nonnegative matrices, instead, conditions can be established in order to guarantee that the power iterates always converge to a global maximizer of $f_A$. The idea is that when the power method is started in the positive orthant then, provided $A$ has an appropriate non-zero pattern, each iterate of the method will stay in this orthant until convergence. Then, a nonlinear Perron-Frobenius type result is proved to guarantee that there exists only one critical point of $f_A$ in this region and this point is a global maximizer of $f_A$. 
 While this idea was already known by Perron himself in the Euclidean $\ell^2$ case, %\cite{Hawkins2008},
 to our knowledge, the first version of this result for norms different than the Euclidean norm, has been proved by Boyd in \cite{Boyd1974}. However, Boyd did not prove the uniqueness of positive critical points but only that they are global maximizer of $f_A$ under the assumption that $A^TA$ is irreducible and $1<p\leq q<\infty$. This work is then revisited by Bhaskara and Vijayaraghavan in \cite{Bhaskara2011} who proved uniqueness for positive matrices $A$ and $1<p\leq q<\infty$. Independently Friedland, Gaubert and Han proved in \cite{Friedland2013} similar results for $1<p\leq 2 \leq q<\infty$ and any nonnegative $A$ such that the matrix $\begin{bsmallmatrix} 0 & A \\ A^T & 0 \end{bsmallmatrix}$ is irreducible. Their result was then extended to $1<p\leq q<\infty$ in \cite{gautier2016tensor} under the assumption that $A^T A$ is irreducible. Finally, all these results have been improved in \cite{GHT_tensors}, leading to the following 
\begin{theorem}[Theorems 3.2 and 3.3, \cite{GHT_tensors}]\label{oldPF}
Let $A\in\RR^{m\times n}$ be a matrix with nonnegative entries and suppose that $A^TA$ has at least one positive entry per row. If $1<p\leq q <\infty$, %and $\norm{\cdot}_{\alpha}=\norm{\cdot}_p$, $\norm{\cdot}_{\beta}=\norm{\cdot}_q$. 
then, every positive critical point of $f_A$ is a global maximizer. Moreover, if either $p<q$ or $A^TA$ is irreducible, then $f_A$ has a unique positive critical point $x^+$ and the power sequence \eqref{eq:power_sequence_boyd} converges to $x^+$ for every positive starting point.
\end{theorem}

% The theorem above is part of Theorems 3.2 and 3.3 in \cite{GHT_tensors}. For the sake of completeness let us point out that Theorems 3.2 and 3.3 in \cite{GHT_tensors} actually provides  more information about $\norm{A}_{q\to p}$ than those stated in Theorem \ref{oldPF},  such as a convergence rate for the case $p<q$ and a number of Gelfand type formulas for $\norm{A}_{q\to p}$. 
%However, the latter results are strongly based on the fact that $\norm{\cdot}_{\alpha}$ and $\norm{\cdot}_{\beta}$ are $\ell^p$ norms and the generalization of their proof to the case of general norms is not trivial and deserves a different analysis which is left open for future work. 
 
 In this work we consider the case of a matrix norm defined in terms of arbitrary vector norms $\|\cdot\|_\alpha$ and $\|\cdot\|_\beta$ and we prove Theorem \ref{newPF} below, which is a new version of Theorem \ref{oldPF}, holding for general vector norms, provided that suitable and mild differentiability and monotonicity conditions are satisfied.
 We stress that Theorems \ref{oldPF} and \ref{newPF} are not corollaries of each other in the sense that there are cases where exactly one, both or none apply. However, when both apply, then Theorem \ref{newPF} is more informative. We discuss in detail these discrepancies in Section \ref{sec:comparison} and give there examples to illustrate them. In particular, a noticeable difference is that, for positive matrices $A$, the newly proposed Theorem \ref{newPF} ensures uniqueness and maximality for choices of $1<p,q<\infty$ that include the range $p>q$. This is, to our knowledge, the first global optimality result for this problem that includes such range of values. %result a sensibly new and particularly relevant  range of values with respect to the current literature.  

 The key of our approach is the use of cone geometry techniques and  the Birkhoff-Hopf theorem, which we recall below.

\section{Cone--theoretic background}\label{sec:cone-background}
We start by recalling concepts from conic geometry. Let $\RR^n_+$ be the nonnegative orthant in $\RR^n$, that is $x\in\RR^n_+$ if $x_i\geq 0$ for every $i=1,\ldots,n$.  The cone $\RR^n_+$ induces a partial ordering on $\RR^n$ as follows: For every $x,y\in \RR^n$ we write $x\leq y$ if $y-x\in\RR^n_+$, i.e. $x_i\leq y_i$ for every $i$. Furthermore, $x,y\in\RR^n_+$ are comparable, and we write $x\sim y$, if there exist $c,C>0$ such that $cy \leq x \leq Cy$. Clearly, $\sim$ is an equivalence relation and the equivalence classes in $\RR^n_+$ are called the parts of $\RR^n_+$. For example, if $n=2$ and $x = (1,0)$, then the equivalence class of $x$ in $\RR^2_+$ is given by $\{(y_1,0): y_1>0\}$. 

For simplicity, from now on we will say that a vector is  nonnegative (resp.\ positive) if its entries are nonnegative (resp.\ positive). The same nomenclature will be used for matrices. 

We recall that a norm $\norm{\cdot}$ on $\RR^n$ is monotonic if for every $x,y\in\RR^n$ such that $|x|\leq |y|$, where the absolute value is taken component-wise, it holds $\norm{x}\leq \norm{y}$ and it is strongly monotonic if for every $x,y\in\RR^n$ with $|x|\neq |y|$ and $|x|\leq |y|$ it holds $\norm{x}<\norm{y}$.

One of the key tools for our main result is the Hilbert's projective metric $d_H\colon \RR^n_+\times \RR^n_+\to [0,\infty]$, defined as follows:
%In 1957, Birkhoff \cite{Bir57} and Samelson \cite{samelson1957perron} independently remarked that one can prove the Perron-Frobenius theorem by using the Hilbert's projective metric $d_H\colon \RR^n_+\times \RR^n_+\to [0,\infty]$ defined as follows:
\begin{equation*}
d_H(x,y)=\begin{cases} \ln\big(M(x/y)M(y/x)\big) & \text{if } x\sim y,\\ 0 & \text{if }x=y=0,\\ \infty, & \text{otherwise}\end{cases}
\end{equation*}
where $M(x/y) = \inf\{C>0 : x \leq Cy\}$.
We collect in the following lemma some useful properties of $d_H$. Most of these results are known and can be found in \cite{lemmens2012nonlinear}. Moreover, similarly to what is observed in Theorem 3 of \cite{gautier2016globally}, we prove a direct relation between the infinity norm and the Hilbert metric, which is useful for deriving explicitly computable convergence rates for the power method.
\begin{lemma}\label{completelemma}
For every $x,y\in\RR^n_+$, it holds $d_H(x,y)=0$ if and only if $x=\lambda y$ for some $\lambda>0$ and $d_H(cx,\tilde{c}y)=d_H(x,y)$ for every $c,\tilde{c}>0$. 
Moreover, let $\norm{\cdot}$ be a monotonic norm on $\RR^n$, $P$ a part of $\RR^n_+$ and define $\metspec=P\cap\{x\in\RR^n_+: \norm{x}=1\}$. Then, $(\metspec,d_H)$ is a complete metric space and
\begin{equation}\label{distequiv}
\norm{x-y}_{\infty} \leq r\,  d_H(x,y) \qquad \forall x,y\in \metspec,
\end{equation}
where $r=\inf\{t>0: x_i\leq t\ \forall x\in \metspec, i=1,\ldots,n\}$.
\end{lemma}
\begin{proof}
Proposition 2.1.1 in \cite{lemmens2012nonlinear} implies that $d_H(x,y)=0$ if and only if $x=\lambda y$ and that $(\metspec,d_H)$ is a metric space. The property $d_H(cx,\tilde{c}y)=d_H(x,y)$ for every $c,\tilde{c}>0$ follows directly from the definition of $d_H$. The completeness of $(\metspec,d_H)$ is a consequence of Proposition 2.5.4 in \cite{lemmens2012nonlinear}. We prove \eqref{distequiv}. If $P=\{0\}$, the result is trivial so we assume $P\neq \{0\}$ and let $i_1,\ldots,i_m$ be such that for any $z\in\RR^n_+$, $z\in P$ if and only if $z_{i_1},\ldots,z_{i_m}>0$. Let $x,y\in\metspec$, then $x\leq M(x/y)y$ and, by monotonicity of $\norm{\cdot}$, it follows 
$1 = \norm{x}\leq M(x/y)\norm{y}= M(x/y).$
Similarly $M(y/x)\geq 1$, so that $M(x/y)M(y/x) \geq \max\big\{M(x/y),M(y/x)\big\}.$ It follows that
\begin{equation*}
d_H(x,y) \geq \ln\big(\max\big\{M(x/y),M(y/x)\big\}\big) = \norm{\overline{x}-\overline{y}}_{\infty},
\end{equation*}
where $\overline{x}=\big(\ln(x_{i_1}),\ldots,\ln(x_{i_m})\big)$ and $\overline{y}=\big(\ln(y_{i_1}),\ldots,\ln(y_{i_m})\big)$. By definition of $r>0$, we have $\ln(x_{i_j}),\ln(y_{i_j})\in (-\infty,\ln(r)]$ for every $j=1,\ldots,m$. Furthermore, by the mean value theorem, we have
\begin{equation*}
|e^s-e^t| \leq |s-t| \max_{\xi \in (-\infty,\ln(r)]}e^{\xi} = r|s-t| \qquad \forall s,t\in (-\infty,\ln(r)].
\end{equation*}
Finally, with $\tilde{x}=(x_{i_1},\ldots,x_{i_m})$ and $\tilde{y}=(y_{i_1},\ldots,y_{i_m})$, we obtain
\begin{equation*}
d_H(x,y) \geq \norm{\overline{x}-\overline{y}}_{\infty} \geq r^{-1} \norm{\tilde{x}-\tilde{y}}_{\infty}=r^{-1}\norm{x-y}_{\infty}\end{equation*} 
which concludes the proof.
\end{proof}

Observe that if $r$ is defined as in Lemma \ref{completelemma} and $\norm{\cdot}$ is strongly monotonic, then
\begin{equation}\label{boundr}
r\leq \tilde r=\max_{i=1,\ldots,n}\frac{1}{\norm{e_i}}\, .
\end{equation}
Indeed, if $y\in\metspec$ is such that there exists $j\in\{1,\ldots,n\}$ with $y_j>\tilde r$, then $1 = \norm{y}>\norm{\tilde r e_j} = \tilde r \norm{e_j}$, which is not possible.

The proof of our main theorem  is based on the Banach contraction principle. 
Thus, for a map $F\colon\RR_+^n\to\RR_+^m$ we consider the Birkhoff contraction ratio $\kappa_H(F)\in[0,\infty]$ of $F$, defined as the smallest Lipschitz constant of $F$ with respect to $d_H$:
\begin{equation*}
\kappa_H(F)=\inf\big\{C>0\ : \ d_H(F(x),F(y))\leq Cd_H(x,y), \ \forall x,y\in\RR^n_{+} \text{ such that } x \sim y\big\}.
\end{equation*}
Clearly, if there exist $x,y\in\RR^n_+$ such that $x\sim y$ and $F(x)\not\sim F(y)$, then $\kappa_H(F)=\infty$. However, such a situation never happens when $F$ is a linear map in which case $\kappa_H(F)\leq 1$ always holds. 
Indeed, if $A\in\RR^{m\times n}$ is a nonnegative matrix, $x,y\in\RR^n_+$ and $x\sim y$, then $x\leq M(x/y)y$ implies $Ax \leq M(x/y)Ay$. Similarly, we have $Ay \leq M(y/x)Ax$ and thus $Ax\sim Ay$. These inequalities also imply that $\kappa_H(A)\leq 1$. 
This upper bound is not tight in many cases. However, thanks to the Birkhoff-Hopf theorem, a better estimate of $\kappa_H(A)$ can be obtained by computing the projective diameter $\Delta(A)\in [0,\infty]$ of $A$, defined as
\begin{equation}\label{projdiam}
\Delta(A)= \sup\big\{d_H(Ax,Ay) : x,y\in\RR^n_+ \text{ with } x\sim y\big\}.
\end{equation}
This is formalized in the following theorem whose proof can be found in Theorems 3.5 and 3.6 of \cite{1995MPCPS}. 
\begin{theorem}[Birkhoff-Hopf] \label{thm:BH}
Let $A\in\RR^{m\times n}$ be a matrix with nonnegative entries, then 
\begin{equation*}
\kappa_H(A)=\tanh\!\big(\Delta(A)/4\big),
\end{equation*}
where $\tanh(t) = (e^{2t}-1)/(e^{2t}+1)$ and with the convention $\tanh(\infty)=1$.
\end{theorem}

The above theorem is particularly useful when combined with the following Theorem 6.2 in \cite{1995MPCPS} and Theorem 3.12 in \cite{Seneta1981}:
\begin{theorem}\label{computeBirkhoff}
Let $A\in\RR^{m\times n}$ be a matrix with nonnegative entries and $e_1,\ldots,e_n$ the canonical basis of $\RR^n$. If there exists $\I\subset \{1,\ldots,n\}$ such that $Ae_i \sim Ae_j$ for all $i,j\in \I$ and $Ae_i=0$ for all $i\notin \I$, then
\begin{equation*}
\Delta(A) = \max_{i,j\in\I}d_H(Ae_i,Ae_j) <\infty.
\end{equation*}
In particular, if all the entries of $A$ are positive, then 
$\Delta(A) = \ln\big(\max_{i,j,k,l}\frac{a_{ki}\, a_{lj}}{a_{kj}\, a_{li}}\big)$ and $\Delta(A)=\Delta(A^T)$.  Moreover, if $A$ has at least one positive entry per row and per column but $A$ is not positive, then $\Delta(A)=\infty$.
\end{theorem}

Unfortunately, such simple formulas for the Birkhoff contraction ratio are, to our knowledge, not known for general nonlinear mappings. We refer however to  Corollary 2.1 in \cite{nussbaum1994} and Corollary 3.9 in \cite{gaubert:hal-00935272} for general characterizations of this ratio.

\section{Nonlinear Perron-Frobenius theorem for $\|A\|_{\beta\to\alpha}$}\label{sec:PF}
%As we will be interested in matrices with nonnegative entries we restrict our attention to real matrices. 
Given  $A\in \RR^{m\times n}$, consider the matrix norm $\|A\|_{\beta\to\alpha}=\max_{x \neq 0}{\|Ax\|_\alpha}/{\|x\|_\beta}$, 
where $\|\cdot\|_\alpha$ and $\|\cdot\|_\beta$ are arbitrary vector norms on $\CC^m$ and $\CC^n$, respectively. Then, as for the case of $\ell^p$ norms, consider the function
% $f_{A}:\CC^n \longrightarrow \RR_+$ defined by
\begin{equation}\label{eq:f_A}
 f_{A}(x) = \frac{\|Ax\|_\alpha}{\|x\|_\beta}\, .
\end{equation}
For an arbitrary possibly non-differentiable vector norm $\|\cdot \|$ it holds (\cite{fletcher2013practical} e.g.)
\begin{equation}\label{eq:norm-subgradient}
 \partial \|x \| = \{y : \<y,x\> = \|x\|, \|y\|_* = 1\}\, ,
\end{equation}
where $\partial$ denotes the subdifferential  %(\cite{Clarke1990,rockafellar1970convex} e.g.), 
and $\|\cdot \|_*$ is the dual norm of $\|\cdot\|$, defined as $\|y\|_* = \max_{x\neq 0}\<x,y\>/\|x\|$. 
Again, for notational convenience, given the vector norm $\|x\|_\alpha$, we introduce the set-valued  operator $J_\alpha$ such that 
% I have outlined this equation because the definition of J_\alpha is very important and should be easy to find.
$$
J_\alpha(x)=\partial \|x\|_\alpha,\quad \forall x\neq 0 \qquad \text{and}\qquad J_{\alpha}(0)=0\, .
$$
The definition of dual norm implies the generalized  H\"older inequality $\<x,y\>\leq \|x\|\|y\|_*$. Thus, for a vector $x$ and a norm $\|\cdot\|_\alpha$, the set of vectors $J_\alpha(x)$ 
coincides with the set of vectors in the unit sphere of the dual norm of $\|\cdot\|_\alpha$, for which equality holds in the H\"older inequality. In fact, the subdifferential of a norm $J_\alpha$ is strictly related with the duality mapping $\mathcal J_\alpha$ induced by that norm. Precisely, by Asplund's theorem (see e.g.\ \cite{cioranescu2012geometry}), we have that 
\begin{equation}\label{eq:asplund}
\mathcal J_\alpha(x) = \frac 1 2 \partial\|x\|_\alpha^2 = \|x\|_\alpha J_\alpha(x)\, .% \{y : \<y,x\> = \|x\|_\alpha, \|y\|_{\alpha*} = \|x\|_\alpha\}\, .
\end{equation}

It is well known that the subgradient of a convex function $f$ is single valued if and only if $f$ is Fr\'echet differentiable. % and, in that case, $\partial f$  coincides with the Fr\'echet derivative $\nabla f$. 
Therefore  $J_\alpha$ is single valued if and only if $\|\cdot\|_\alpha$ is a  Fr\'echet differentiable norm. 
The assumption that the duality maps involved are single valued will be crucial for our main result. For this reason, throughout we make the following assumptions on the norms we are considering 
\begin{assumption}\label{a}
The norms $\|\cdot\|_\alpha$ and $\|\cdot\|_\beta$ we consider are such that 
\begin{enumerate}
    \item $\|\cdot\|_\alpha$ is Fr\'echet differentiable. \label{a1}
    \item The dual norm $\|\cdot\|_{\beta^*}$ is Fr\'echet differentiable. \label{a2}
    \item Both $\|\cdot\|_\alpha$ and $\|\cdot\|_{\beta*}$ are strongly monotonic.  \label{a3}
\end{enumerate}
\end{assumption}
\begin{remark}Recall that every monotonic norm $\norm{\cdot }$ is also absolute (see e.g.\ \cite[Thm.~1]{JOHNSON199143}), that is $\|\,|x|\,\|=\norm{x}$
for every $x$, where $|x|$ denotes the entrywise absolute value. This implies, in particular, that a monotonic norm is Fr\'echet differentiable at every $x\in\RR^n\saufzero$ if and only if it is Fr\'echet differentiable at every $x\in\RR^n_+\saufzero$.
\end{remark}
% \begin{NEW}
% \begin{remark}
% Note that the role of $\|\cdot\|_\alpha$ and $\|\cdot\|_\beta$ in Assumption \ref{a} is somewhat symmetric. In fact, due to the identity $\|A\|_{\beta\to\alpha}=\|A^T\|_{\alpha^*\to\beta^*}$, if we exchange the role of $\alpha$ and $\beta$ in Assumption \ref{a}, the remainder of this work 
% \end{remark}
% \end{NEW}
Points \eqref{a1} and \eqref{a2} of Assumption \ref{a} ensure that the following nonlinear mapping 
\begin{equation}\label{eq:SA}
 \SS_A(x)=J_{\beta^*}(A^T J_\alpha( Ax))
\end{equation}
is single valued. Point \eqref{a3} ensures that for nonnegative matrices the maximum of $f_A$ is attained on a nonnegative vector %. The strong monotonicity further implies 
and that if $A^\top A$ is irreducible, then this maximizer has positive entries. Overall, they allow us to prove the following fundamental preliminary Lemmas \ref{lem:crazy}--\ref{lemairr}. 

First, we discuss the critical points of $f_A$. If $\|\cdot\|_\alpha$ and $\|\cdot\|_\beta$ satisfy Assumption \ref{a}, then $f_A$ may not be differentiable. Indeed, the differentiability of $\norm{\cdot}_{\beta^*}$ does not imply that of $\norm{\cdot}_{\beta}$ (see for instance \cite[Chapter II]{cioranescu2012geometry}). Hence, in the following, we use Clarke's generalized gradient \cite{Clarke1990} to discuss the critical points of $f_A$. In particular, let us recall that, by \cite[Prop. 2.2.7]{Clarke1990}, the generalized gradient of a convex function coincides with its subgradient. Moreover, it can be verified that $f_A$ is locally Lipschitz near every $x\in \RR^n\saufzero$ so that its generalized gradient $\partial f_A(x)\subset \RR^n$ is well defined and $x$ is a critical point of $f_A$ if $0\in \partial f_A(x)$. Moreover, if $f_A$ attains a local minimum or maximum at $x\neq 0$, then $0\in\partial f_A(x)$ by \cite[Prop. 2.3.2]{Clarke1990}.

\begin{lemma} \label{lem:crazy}
Let $\|\cdot\|_\alpha$, $\|\cdot\|_\beta$ satisfy Assumption \ref{a} and let $x\in \RR^n_+$ with $\|x\|_{\beta}=1$ and $f_A(x)\neq 0$. If $x$ is a critical point of $f_A$, then it is a fixed point of $\SS_A$. Conversely, if $x$ is a fixed point of $\SS_A$ and $\|\cdot\|_{\beta}$ is differentiable, then $x$ is a critical point of $f_A$.
\end{lemma}
\begin{proof}
First, assume that $0\in \partial f_A(x)$. As $\norm{\cdot}_{\alpha}$ and $\norm{\cdot}_{\beta}$ are Lipschitz functions and $\norm{x}_{\beta}=1$, Proposition 2.3.14 of \cite{Clarke1990} implies that 
\begin{equation}\label{quotientdiff}
\partial f_A(x)\subset A^TJ_{\alpha}(Ax)-f_A(x)J_{\beta}(x).
\end{equation}
$J_\alpha$ is single valued since $\|\cdot\|_{\alpha}$ is differentiable. Hence, $0\in \partial f_A(x)$ implies that $f_A(x)^{-1}A^TJ_{\alpha}(Ax)\in J_{\beta}(x)$. Now, as $\|\cdot\|_{\beta^*}$ is differentiable, we have that, for the duality mapping $\mathcal J_\beta$, it holds $y \in \mathcal J_\beta(x)$ if and only if $x = \mathcal J_{\beta^*}(y)$ (c.f. \cite[Prop.\ 4.7]{cioranescu2012geometry}). It follows, with \eqref{eq:asplund}, that $ \lambda J_{\beta^*}(A^TJ_\alpha(Ax))= x$ with $\lambda >0$. Finally, as $\norm{J_{\beta^*}(A^TJ_\alpha(Ax))}_{\beta}=1=\norm{x}_{\beta}$, we have $\lambda = 1$ which implies that $\SS_A(x)=x$. 

Now, suppose that $x$ is a fixed point of $\SS_A$. Then, we have $J_{\beta}(\SS_A(x))=J_{\beta}(x)$. Again, by \cite[Prop.\ 4.7]{cioranescu2012geometry} and \eqref{eq:asplund}, we deduce the existence of $\lambda>0$ such that $ \lambda\,A^\top J_{\alpha}(Ax)\in J_{\beta}(x)$. The definition of $J_{\beta}$ implies that $\ps{x}{\lambda\,A^\top J_{\alpha}(Ax)} = \norm{x}_{\beta}=1$ and thus $\lambda^{-1}=\ps{Ax}{AJ_{\alpha}(Ax)}=f_A(x)$. It follows that $0\in A^\top J_{\alpha}(Ax)-f_A(x)J_{\beta}(x)$. If $\norm{\cdot}_{\beta}$ is differentiable, then $f_A$ is differentiable at $x$ and the sets in \eqref{quotientdiff} are equal (and singletons). It follows that $0\in \partial f_A(x)$.
\end{proof}

\begin{lemma}\label{PFmono}
Let $A\in\RR^{m\times n}$ be a matrix with nonnegative entries and $P$, a part of $\RR^n_+$ such that $A^TAx\in P$ for every $x\in P$. Furthermore, let  $\norm{\cdot}_{\alpha}$ and $\norm{\cdot}_{\beta}$ satisfy Assumption \ref{a}. % strongly monotonic differentiable norms on $\RR^{n}$ and $\RR^m$, respectively.
If $\kappa_H(\SS_A)\leq \tau<1$, 
then $\SS_A$ has a unique fixed point $z$ in $P$ and for every positive integer $k$ and every $x\in P$, it holds 
$$\norm{\SS^k_{A}(x)-z}_{\infty}\leq  \tau^k \, (r/(1-\tau)) \, d_H(x,\SS_{A}(x))%(1-\tau(\SS_A))^{-1}
,$$
where $r=\inf\{t>0: x_i\leq t\ \forall i=1,\ldots,n, \ x\in P, \norm{x}_{\beta}=1\}$.
\end{lemma}
\begin{proof}
By assumption %definition of $\SS_A$, it is clear that $\kappa_H(\SS_A)\leq \tau(\SS_A)$. Hence, 
$\SS_A$ is a strict contraction on the metric space $(\metspec,d_H)$ where $\metspec=P\cap\{x \in \RR^n_+: \norm{x}_{\beta}=1\}$. As $(\metspec,d_H)$ is complete by Lemma \ref{completelemma}, it follows from the Banach fixed point theorem (see for instance Theorem 3.1 in \cite{pointfixe}) that $\SS_A$ has a unique fixed point $z$ in $\metspec$ and for every $y\in \metspec$ it holds
\begin{equation*}
d_H(\SS^k_A(y),z) \leq \frac{\tau^k}{1-\tau} d_H(y,\SS_A(y)) \, .
\end{equation*}
As $\SS_A(\lambda y)=\SS_A(y)$ and $d_H(\lambda y, \SS_A(y))=d_H(y, \SS_A(y))$ for every $\lambda >0$,    the convergence rate is a direct consequence of the above inequality and Lemma \ref{completelemma}.
\end{proof}

We remark that this result does not guarantee that the unique fixed point $z$ of $\SS_A$ in $P$ is a global maximizer of $f_A$ and in fact this is not always true. Indeed, if $A$ is a $2\times 2$ diagonal matrix which is not a multiple of the identity and $\norm{\cdot}_\alpha=\norm{\cdot}_2$, $\norm{\cdot}_{\beta}=\norm{\cdot}_3$, then $\kappa_H(\SS_A)\leq 1/2$ and $\SS_A$ leaves all the parts of $\RR^2_+$ invariant but some of them do not contain a global maximizer of $f_A$. Moreover, as $\RR^n_+$ has $2^n$ parts, testing each part of the cone is computationally too expensive for large $n$. Therefore, in the remaining part of the section, we derive conditions in order to ensure that the power iterates  converge to a global maximizer of $f_A$. 
%

%First of all, we note that for nonnegative matrices, the global maximum of .... is attained in $\RR^n_+$ when $A$ has nonnegative entries and both the norms $\|\cdot\|_\alpha$ and $\|\cdot\|_\beta$ are monotonic.
\begin{lemma}\label{nonnegmax}
Let $A\in \RR^{m\times n}$ be a matrix with nonnegative entries and let $\norm{\cdot}_\alpha$, $\norm{\cdot}_\beta$ satisfy Assumption \ref{a}. %be monotonic norms on $\CC^m$ and $\CC^n$ respectively. 
Then it holds $f_A(x)\leq f_A(|x|)$ for any $x\in \CC^n\setminus\{0\}$ and the maximum of $f_A$ is attained in $\RR^n_+$. %\red{Moreover, for every $x\in \RR^n_+$, it holds  $\SS_A(x)\sim A^\top A \, x$.}
\end{lemma}
\begin{proof}
Let $x\neq 0$, since $A$ has nonnegative entries, it holds $|Ax|\leq A|x|$. Thus, as monotonic norms are also absolute, we have 
\begin{equation*}
f_A(x)=\frac{\norm{Ax}_{\alpha}}{\norm{x}_{\beta}}=\frac{\norm{|Ax|}_{\alpha}}{\norm{|x|}_{\beta}}\leq\frac{\norm{A|x|}_{\alpha}}{\norm{|x|}_{\beta}}=f_A(|x|).
\end{equation*}
Now, if $y$ is a global maximizer of $f_A$, then $f_A(y)\leq f_A(|y|)\leq f_A(y)$ which concludes the proof.
\end{proof}

In the forthcoming Lemma \ref{lemairr}, we use the strong monotonicity required in Point \eqref{a3} of Assumption \ref{a} to prove that if $A^\top A$ is irreducible, then the nonnegative maximizer of Lemma \ref{nonnegmax} has positive entries. To this end, however, we need one additional preliminary result that characterizes strongly monotonic norms in terms of the zero pattern of $J$ and which we prove in the following: % This result will follow from the following lemma which characterize strongly monotonic norms in terms of the zero pattern of $J$.
\begin{lemma}\label{patternJ}
Let $\norm{\cdot}_{\gamma}$ be a differentiable monotonic norm on $\RR^n$, then $\norm{\cdot}_{\gamma}$ is strongly monotonic if and only if $x\sim J_{\gamma}(x)$ for every $x\in\RR^n_+$.
\end{lemma}
\begin{proof}
Suppose that $\norm{\cdot}_{\gamma}$ is strongly monotonic. Let $x\in\RR^n_{+}$. If $x=0$, $J_{\gamma}(0)=0$ by construction. Suppose that $x\neq 0$. We use the strong monotonicity to prove the existence of $c>0$ such that $c\,x\leq J_{\gamma}(x)$. Let $i$
 be such that $x_i>0$ and define $f(t)=\norm{x+(t-x_i)e_i}_{\gamma}$ for all $t>0$. Then, $f$ is differentiable and $f'(t)=J_{\gamma}(x+(t-x_i)e_i)_i$ for all $t>0$. Furthermore, $f$ is strictly increasing on $(0,\infty)$ since $\norm{\cdot}$ is strongly monotonic. It follows that $J_{\gamma}(x)_i = f'(x_i) >0$. As this is true for all $i$ such that $x_i>0$, we conclude that there exists $c >0$ such that $c\,x\leq J_{\gamma}(x)$. The existence of $C >0$ such that $J_{\gamma}(x)\leq C\,x$ follows from 
 Proposition 5.2 of \cite[Chapter 1] {cioranescu2012geometry}. Hence, we have $J_{\gamma}(x)\sim x$.

For the reverse implication, suppose that  $J_{\gamma}(x)\sim x$ for all $x\in\RR^n_+$. Let $x,y\in \RR^n_+$ be such that $x\leq y$ and $x\neq y$. If $x=0$, then $\norm{x}_{\gamma}=0<\norm{y}_{\gamma}$. Suppose that $x\neq 0$. As $x\leq y$ and $x\neq 0$, there exists $i$ and $t_0>0$ such that $x+te_i \leq y$ for all $t\in (0,t_0)$. For $t\in (0,t_0)$, we have 
$$\norm{y}_{\gamma}\geq\norm{x+\tfrac{1}{2}(t_0+t)e_i}_{\gamma}\geq \norm{x+\tfrac{t_0}{2}e_i}_{\gamma}+\ps{J_{\gamma}(x+\tfrac{t_0}{2}e_i)}{\tfrac{t}{2}e_i}\geq \norm{x}_{\gamma}+\tfrac{t}{2}J_{\gamma}(x+\tfrac{t_0}{2}e_i)_i,$$
where the second inequality follows from the convexity of $\norm{\cdot}_{\gamma}$. By assumption, we have $J_{\gamma}(x+\tfrac{t_0}{2}e_i)\sim x+\tfrac{t_0}{2}e_i$ and thus $J_{\gamma}(x+\tfrac{t_0}{2}e_i)_i>0$. It follows that $\norm{y}_{\gamma} > \norm{x}_{\gamma}$, i.e. $\norm{\cdot}_{\gamma}$ is strongly monotonic.
\end{proof}
\begin{lemma}\label{lemairr}
Let $\|\cdot\|_{\alpha}$ and $\|\cdot\|_{\beta}$ satisfy Assumption \ref{a}. Let $A$ be a matrix with nonnegative entries and suppose that $A^TA$ is irreducible. Then, $\SS_A(x)$ is positive for every positive $x$ and 
%there exists a positive integer $k$ such that for every $y\in\RR^n_+\setminus\{0\}$, $\SS^{k}_A(y)$ is positive. Moreover, 
every nonnegative critical point of $f_A$ is positive.
\end{lemma}
\begin{proof}
Lemma \ref{patternJ} implies that $\SS_A(x) \sim A^\top A x$. It follows that $\SS_A$ maps positive vectors to positive vectors since the irreducibility of $A^\top A$ implies that $A^\top A$ is positive for all positive $x$. Finally, note that $A^\top A$ is symmetric positive semi-definite and therefore all its eigenvalues are nonnegative. It follows that $A^\top A$ is primitive (see e.g.\ Theorem 1 in \cite{Francesco}). 
%First of all, note that for every $v\in\RR^n$, it holds $\langle A^TAv,v\rangle = \norm{Av}_2\geq 0$ and thus 
 By the same theorem, there exists a positive integer $k$ such that $(A^TA)^k$ is a matrix with positive entries. Since $\SS^{k}_A(x)\sim (A^TA)^{k}x$ for every $x\in \RR^n_+\setminus\{0\}$, we deduce that $\SS^{k}_A(x)$ is strictly positive for every nonzero, nonnegative $x$. Finally, suppose that $y\in\RR^n_+$ is a critical point of $f_A$, then $y$ is a fixed point of $\SS_A$ by Lemma \ref{lem:crazy} and thus $y=\SS_A^k(y)$ is strictly positive.
\end{proof}

We are now ready to state  our main theorem of this section. This theorem provides conditions on $A$, $\norm{\cdot}_{\alpha}$ and $\norm{\cdot}_{\beta}$ that ensure the existence of a unique positive maximizer $x^+$ such that $\|Ax^+\|_\beta/\|x^+\|_\alpha=\|A\|_{\beta\to\alpha}$ and that govern the convergence of the power sequence 
\begin{equation}\label{eq:power_sequence}
x_0  = x_0 /\|x_0 \|_{\beta}, \quad x_{k+1} = J_{\beta^*}(A^TJ_\alpha(Ax_k))\qquad \text{ for } k=0,1,2,3,\dots
\end{equation}
to such $x^+$. As announced, this result is essentially a fixed point theorem for $\SS_A$ and thus the Birkhoff contraction ratio $\kappa_H(\SS_A)$ and any $\tau$ that well-approximate $\kappa_H(\SS_A)$ from above  play a  central role. 
%
%
% We eventually obtain our main global convergence theorem:
\begin{theorem}\label{newPF}
Let $A\in\RR^{m\times n}$ be a matrix with nonnegative entries and suppose that $A^TA$ is irreducible. Let $\norm{\cdot}_{\alpha}$ and $\norm{\cdot}_{\beta}$ satisfy Assumption \ref{a}. 

If $\kappa_H(\SS_A)\leq \tau <1$, then:
\begin{enumerate}
 \item $f_A$ has a unique critical point $x^+$ in $\RR^n_+$. Moreover, $f_A(x^+)=\norm{A}_{\beta\to\alpha}$ and  $x^+$ is positive.
  \item If $x_0$ is positive and $x_{k+1}=\SS_A(x_k)$ is the power sequence, then 
  \begin{equation*}
 \norm{x_k-x^+}_{\infty}\leq \tau^k\,C\, \quad \text{with}\quad C=\max_{i=1,\ldots,n}\frac{d_H(x_0,x_1)}{(1-\tau)\norm{e_i}_{\beta}} 
 \end{equation*}
 where $e_1,\ldots,e_n$ is the canonical basis of $\RR^n$. Furthermore, it holds
 \begin{equation*}
 (1-\tau^k\,\tilde C)\norm{A}_{\beta \to \alpha}\, \leq \, \norm{Ax_k}_\alpha \, \leq \, \norm{A}_{\beta\to \alpha}
 \end{equation*}
with $\tilde C=\displaystyle C\,\max_{x\neq 0}\tfrac{\norm{x}_{\alpha}}{\norm{x}_{\infty}}$.
 In particular, $x_k\to x^+$ as $k\to \infty$.
 %\item If $x_0$ is positive and $x_k$ is the sequence defined by (NNPM), then it holds $f_A(x_k)\to \|A\|_{\beta\to\alpha}$ as $k\to \infty$ and $x^+$ is the limit of every convergent subsequence of $x_k$.
\end{enumerate}
\end{theorem}
\begin{proof}
Lemma \ref{nonnegmax} implies that $f_A$ has a maximizer $x^+\in\RR^n_+$. Lemma \ref{lemairr} implies that $x^+$ is positive and that the interior of $\RR^n_+$ is left invariant by $\SS_A$. Hence, all statements except the bounds on $\norm{Ax_k}_{\alpha}$
follow by a direct application of Lemma \ref{PFmono} and Equation \eqref{boundr}. We conclude with a proof of the estimates for $\norm{Ax_k}_{\alpha}$. Clearly, $\norm{Ax_k}_{\alpha}\leq \norm{A}_{\beta\to \alpha}$ always hold. For the lower bound, let $\gamma = \max_{x\neq 0}\tfrac{\norm{x}_{\beta}}{\norm{x}_{\infty}}$. The estimate on $\norm{x_k-x^+}_{\infty}$ implies that \begin{align*}
\norm{A}_{\beta\to\alpha}&-\norm{Ax_k}_{\alpha}=\norm{Ax^+}_{\alpha}-\norm{Ax_k}_{\alpha}\leq \norm{A(x^+-x_k)}_{\alpha}\\
&\leq \norm{A}_{\beta\to\alpha}\norm{x^+-x_k}_{\beta}
\leq \gamma\,\norm{A}_{\beta\to\alpha}\norm{x^+-x_k}_{\infty}\leq \tau^k\,C\,\gamma\,\norm{A}_{\beta\to\alpha}
\end{align*}
which concludes the proof.
\end{proof}

%\subsection{A sharp upperbound for $\kappa_H(\SS_A)$}
Theorem \ref{newPF} holds for any upperbound $\tau$ of $\kappa_H(\SS_A)$ and a somewhat natural choice for such a $\tau$ is the following 
\begin{equation}\label{deftau}
 \tau(\SS_A) = \kappa_H(A^T)\kappa_H(J_{\beta^*})\kappa_H(A)\kappa_H(J_{\alpha})\, .
\end{equation}
This coefficient is particularly useful in practice as, thanks to the Birkhoff-Hopf theorem, in many circumstances one can provide explicit bounds for $\tau(\SS_A)$. Although in principle $\tau(\SS_A)$ can be larger than $\kappa_H(\SS_A)$, in  the forthcoming Section \ref{sharpsec} we show that there are cases where the equality $\tau(\SS_A)=\kappa_H(\SS_A)$ holds. Moreover, we discuss the sharpness of the condition $\kappa_H(\SS_A)<1$ required by our main result. % in Section \ref{sharpsec}. 
In the following Section \ref{sec:comparison}, instead, we discuss the particular case where $\norm{\cdot}_{\alpha},\norm{\cdot}_\beta$ are $\ell^p$ norms and we give examples showing how Theorem \ref{newPF} improves the existing theory for this problem.
\subsection{Examples and comparison with previous work}\label{sec:comparison}
When $\norm{\cdot}_{\alpha}$ and $\norm{\cdot}_{\beta}$ are $\ell^p$ norms,  Theorem \ref{newPF} implies the following:

\begin{corollary}\label{pqPF}
Let $A\in\RR^{m\times n}$ be a matrix with nonnegative entries and suppose that $A^\top A$ is irreducible. Let $1<p,q<\infty$ and consider $$\norm{A}_{q\to p}=\max_{x\neq 0}\frac{\norm{Ax}_p}{\norm{x}_q}, \qquad \text{and}\qquad\tau = \kappa_H(A)\kappa_H(A^\top)\frac{p-1}{q-1}.$$
If $\tau <1$, then $\norm{A}_{q\to p}$ can be approximated to an arbitrary  precision with the fixed point iteration \eqref{eq:power_sequence_boyd}.
\end{corollary}

% \begin{proof}
% Besides the complexity bound, the result is a direct consequence of Theorem \ref{newPF}. So let us compute the total number of operations required by the power sequence \ref{eq:power_sequence}.
% As $J_q(x)=\norm{x}_q^{1-q}\Phi_q(x)$ and $J_p(y)=\norm{x}_p^{1-p}\Phi_p(x)$, $J_p$ and $J_q$ can be respectively evaluated in $\mathcal{O}(m)$ and $\mathcal{O}(n)$ arithmetic operations. Hence, $\SS_A$ can be evaluated in $\mathcal{O}(m \, n)$ operations. 
% Now, let $\tilde C$ be as in Theorem \ref{newPF}. We have $\tilde C\tau^k <\epsilon$ if and only if $k>(\ln(\epsilon)-\ln(\tilde C))/\ln(\tau)$. 
% As $(\ln(\epsilon)-\ln(\tilde C))/\ln(\tau)\in\mathcal{O}(-\ln(\epsilon))$ for $\epsilon\to 0$, we deduce that $\norm{A}_{q\to p}-\epsilon\leq\norm{Ax_k}_{p}$ after $\mathcal{O}(\ln(\epsilon^{-1}))$ iterations of $\SS_{A}$ leading to a total complexity of $\mathcal{O}(m\, n\,\ln(\epsilon^{-1}))$ which concludes the proof.
% \end{proof}
In the case of $\ell^p$ norms, both Theorem \ref{oldPF} and Corollary \ref{pqPF} apply.  In order to compare them let us compute the Birkhoff contraction ratio for some simple but explanatory cases. 
Let $\epsilon \geq 0$ and $A\in\RR^{3\times 2}$, $B\in\RR^{2\times 2}$, $C\in\RR^{3\times 3}$ be defined as 
\begin{equation*}
A=\begin{bmatrix} 1 & 2 \\ 3 & 4 \\0 & 0 \end{bmatrix}, \qquad B = \begin{bmatrix} \epsilon & 1 \\ 1 & \epsilon \end{bmatrix}, \qquad C = \begin{bmatrix} 0 & 1 & 1 \\ 2 & 2 &2 \\ 3 & 3 & 0 \end{bmatrix}.
\end{equation*} 
Due to Theorem \ref{computeBirkhoff}, it is easy to see that %$$ and  
\begin{align*}
%\begin{array}{lcl}
&\kappa_H(A) =\tanh(3/8)\leq 9/25, \\
&\kappa_H(A^T)= \tanh(1/16)\leq 1/16,\\
&\kappa_H(B)=\kappa_H(B^T)=(1-\epsilon)^2/(1+\epsilon)^2\\
& \kappa_H(C)=\kappa_H(C^T)= 1\, . 
%\end{array}
\end{align*}
 Note that $A^TA$ and $C^TC$ are positive matrices and $B^TB$ is positive if and only if $\epsilon>0$. If $\epsilon=0$, then $B^T B$ is the identity matrix. We first discuss the implications of Theorem \ref{oldPF} for the computation of $\norm{X}_{q\to p}$ where $X\in\{A,B,C\}$.

 If $p\leq q$ and $\epsilon>0$, then Theorem \ref{oldPF} implies that $f_X$ has a unique positive maximizer $x^+$, which is global, and the power sequence \eqref{eq:power_sequence} will converge to $x^+$. However, if $\epsilon=0$ then Theorem \ref{oldPF} ensures that every positive critical point of $f_B$ is a global maximizer but uniqueness and convergence are only guaranteed under the assumption $p<q$. Now, we look at the implications of Theorem \ref{newPF}. By noting that $\kappa_H(J_p) = p-1$ and $\kappa_H(J_{q^*}) = 1/(q-1)$, we have
\begin{equation*}
\tau(\SS_A)\leq \frac{9}{400}\,\, \frac{p-1}{q-1}, \qquad \tau(\SS_B) =\left(\frac{1-\epsilon}{1+\epsilon}\right)^2\frac{p-1}{q-1},\qquad \tau(\SS_C)=\frac{p-1}{q-1}.
\end{equation*}
Hence, for instance, uniqueness and global maximality of a positive maximizer of $f_A$ is guaranteed by Theorem \ref{newPF} under the assumption $9(p-1)<400(q-1)$ which includes the known global convergence range of values $p<q$, but is of course a much weaker assumption.  

Now, note that for $\epsilon \geq 1$ we have $\tau(\SS_B)<1$ if and only if $(\epsilon-1)^2(p-1)<(\epsilon+1)^2(q-1)$. This assumption is less restrictive than $p\leq q$ for every $\epsilon \geq 1$ as $p\leq q$ correspond to the asymptotic case $\epsilon \to \infty$. If $\epsilon =1$, Theorem \ref{newPF} applies for every $1<p,q<\infty$.
The analysis for $0<\epsilon < 1$ is similar. However, we note that if $\epsilon=0$, then Theorem \ref{newPF} does not provide any information about $f_B$ for the case $p=q$ in contrast with Theorem \ref{oldPF}. When $\epsilon=0$ and $p<q$, both theorems imply the same result. Finally, note that $\tau(\SS_C)<1$ if and only if $p<q$ and so Theorem \ref{oldPF} is more useful as it also covers the case $p=q$.

More in general, when the considered matrix $A$ has finite projective diameter $\Delta(A)$, then Theorem \ref{thm:BH} implies that $\kappa_H(A)<1$ and thus Theorem \ref{newPF} ensures that for any $p>1$,  the matrix norm $\|A\|_{q\to p}$ can be approximated in polynomial time to an arbitrary precision for any choice of $q>\kappa_H(A)^2(p-1)+1$, without the requirement $q>p$.

Figure \ref{fig:p(A)} shows that the value of $\kappa_H(A)$ for matrices with positive entries is often substantially smaller than one, enhancing the relevance of Theorem \ref{newPF}.

\begin{figure}[!t] 
\centering
\includegraphics[width=.9\textwidth]{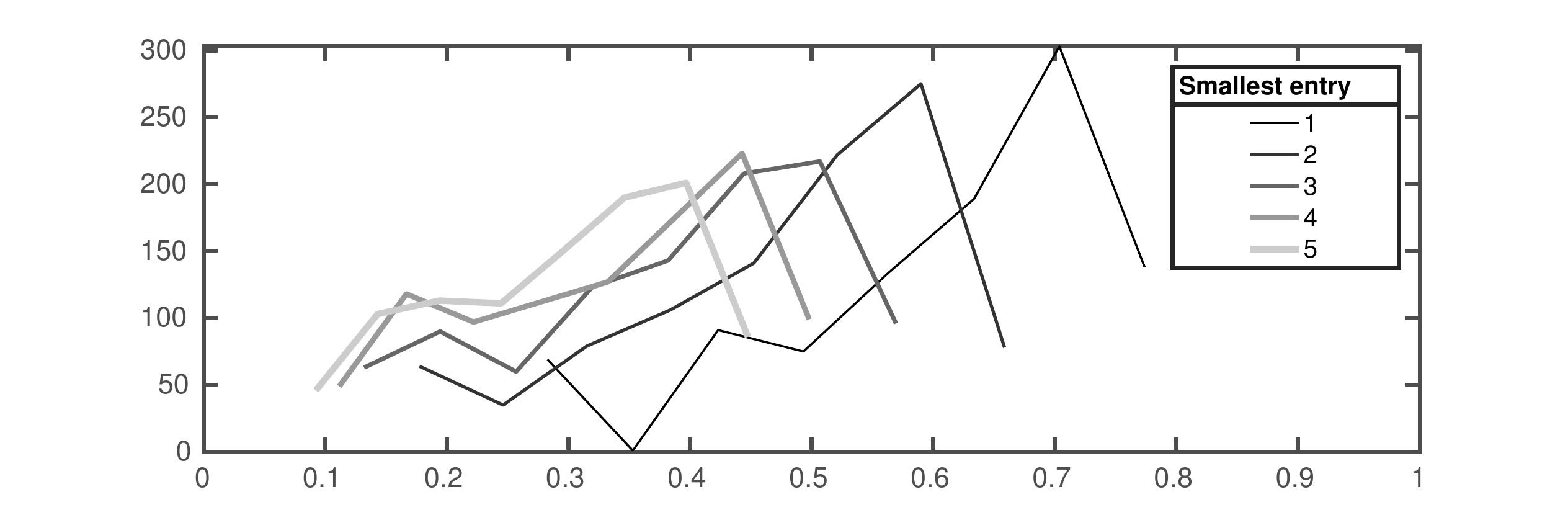}
\caption{Each line shows the distribution of $\kappa_H(A)$ over 1000 random matrices $A\in\RR_+^{10\times 10}$  with entries between $k$ and $10$. Different curves correspond to different values of $k \in \{1,2,\dots, 5\}$. }\label{fig:p(A)}
% \caption{Boxplot showing the distribution  of $\kappa_H(A)$ over 1000 random matrices $A\in\RR_+^{50\times 50}$ with entries between $k$ and $50$. Different boxes correspond to different values of $k \in \{1,4,7,\dots, 25\}$. }\label{fig:p(A)}
\end{figure}

% In the next section, we study the sharpness of the condition $\tau<1$ in Theorem \ref{pqPF}.

\subsection{On the sharpness of the new convergence condition}\label{sharpsec}
As we observed earlier, the key property behind the global convergence of the power iterates relies on the fact that, when $\kappa_H(\SS_A)<1$, the mapping $\SS_A$ has a unique positive fixed point $x^+$.  Due to Lemma \ref{lem:crazy}, this is equivalent to observing that, in this case, $x^+$ is the unique positive critical point of $f_A$, up to scalar multiples.  In what follows we show that this is not anymore the case if $\kappa_H(\SS_A)>1$. In particular, we limit our attention to the case of $\ell^p$ norms %$f_A(x) = \|Ax\|_p/\|x\|_q$ 
and we exhibit a one-parameter family of $2\times 2$ positive and symmetric matrices $A_\epsilon$ for which a unique positive critical point of $f_{A_\epsilon}$ exists if and only if $\kappa_H(\SS_{A_\epsilon})\leq 1$. Moreover, we show that for such a family of matrices it holds $\tau(\SS_A)=\kappa_H(\SS_A)$ where $\tau(\SS_A)$ is the estimate of $\kappa_H(\SS_A)$ discussed in equation \eqref{deftau}. 
As $f_A$ is scale invariant, here and in the rest of this section, uniqueness  of the critical point is meant up to scalar multiples. %  For the sake of simplicity, for the rest of this section we shall assume this For the sake of simplicity we shall work on

For $\epsilon>0$ and $p,q\in(1,\infty)$, let $A_{\epsilon}\in\RR^{2\times 2}$ and $f_{A_\epsilon}:\RR^{2}\to \RR_+$ be defined as
\begin{equation*}
A_{\epsilon}=\begin{bmatrix} \epsilon & 1 \\ 1 & \epsilon \end{bmatrix} \qquad \text{and}\qquad f_{A_\epsilon}(x)=\frac{\|A_{\epsilon}x\|_p}{\|x\|_q}.
\end{equation*}
The main result of this section is the following theorem, whose proof is postponed to the end of the section
\begin{theorem}\label{thm:thightness}
%We have $\tau(\SS_{A_{\epsilon}})=\big(\frac{1-\epsilon}{1+\epsilon}\big)^2\frac{}{p-1}{q-1}$ and 
It holds $\kappa_H(\SS_{A_{\epsilon}})=\tau(\SS_{A_{\epsilon}})$. Furthermore, $f_{A_\epsilon}$ has a unique critical point in $\RR_+^2$ %on the $\|\cdot\|_q$ sphere 
if and only if $\tau(\SS_{A_\epsilon})\leq~1$.
\end{theorem}
This result shows that, unlike the previous Theorem \ref{oldPF}, Theorem \ref{newPF} is tight in the sense that when  $\kappa_H(\SS_A)>1$ there might be multiple distinct fixed points of $\SS_A$ in $\RR^2_+$, and thus convergence of the power sequence to a prescribed fixed point cannot be ensured globally without restrictions on the starting point $x_0\in\RR^2_+$. %The extension of Theorem \ref{newPF} to the non-expansive case $\tau(\SS_A)=1$ for general matrices can be addressed by exploiting different forms of irreducibility for the nonlinear map $\SS_A$ (cf. \cite[\S 3--\S 5]{gautier2017theorem}), but we leave this case for future work. 

We subdivide the proof of Theorem \ref{thm:thightness} above into a number of preliminary results. Before proceeding, we recall that for $p\in(1,\infty)$, $\Phi_p\colon \RR^n\to \RR^n$ is entrywise defined as $\Phi_p(x)_i = |x_i|^{p-2}x_i$ for all $i$. We compute $\tau(\SS_{A_\epsilon})$ and $\kappa_H(\SS_A)$.
\begin{lemma}\label{lem:tau(Aeps)}
For every $\epsilon >0$, we have 
$\kappa_H(\SS_{A_{\epsilon}})=\tau(\SS_{A_{\epsilon}})=\big(\frac{1-\epsilon}{1+\epsilon}\big)^2\frac{p-1}{q-1}.$
\end{lemma}
\begin{proof}
As $\kappa_H(A)=\big| \frac{\epsilon-1}{1+\epsilon}\big|$ by Theorem \ref{computeBirkhoff}, we have $\tau(\SS_{A_\epsilon})=\big(\frac{1-\epsilon}{1+\epsilon}\big)^2\frac{p-1}{q-1}$. Now, we show that $\kappa_H(\SS_{A_\epsilon})=\tau(\SS_{A_\epsilon})$. Clearly, $\kappa_H(\SS_{A_\epsilon})\leq\tau(\SS_{A_\epsilon})$, for the reverse inequality consider $x=(1,1)^T$ and $y(t)=(1,t)^T$. Furthermore, define $h\colon (1,\infty)\to \RR$ as
$$h(t) = \frac{d_H\big(\SS_{A_\epsilon}(x),\SS_{A_\epsilon}(y(t))\big)}{d_H\big(x,y(t)\big)}.$$
Then, we have $h(t)\leq \kappa_H(\SS_{A_\epsilon})$ for every $t>0$. To conclude the proof, we show that $\lim_{t\to 1^+}h(t)=\tau(\SS_{A_{\epsilon}})$. A direct computation shows that
$d_H\big(x,y(t)\big)=\ln(t)$ and $A_{\epsilon}\Phi_p(A_{\epsilon}x) =(1+\epsilon)^p(1,1)^T$. Recalling that $\SS_{A_{\epsilon}}(z)= \Phi_{q^*}(A_{\epsilon}\Phi_p(A_{\epsilon}z))$, we have
$$d_H(\SS_{A_\epsilon}(x),\SS_{A_\epsilon}(y(t)))=(q^*-1)d_H(A_{\epsilon}\Phi_p(A_{\epsilon}x),A_{\epsilon}\Phi_p(A_{\epsilon}y(t))).$$
So if we let $f_1,f_2\colon(1,\infty)\to \RR$ be such that $A_{\epsilon}\Phi_p(A_{\epsilon}y(t))=\big(f_1(t),f_2(t)\big)^T$ for all $t>1$, we get 
$$
\exp\Big((q-1) d_H(\SS_{A_\epsilon}(x),\SS_{A_\epsilon}(y(t)))\Big)=\max\Big\{\frac{f_1(t)}{f_2(t)},\frac{f_2(t)}{f_1(t)}\Big\}.
$$
With $$g(t)=\frac{f_1(t)}{f_2(t)}=\frac{\epsilon  (t+\epsilon )^{p-1}+(t \epsilon +1)^{p-1}}{(t+\epsilon )^{p-1}+\epsilon  (t \epsilon +1)^{p-1}},$$ the above computations, imply
$$(q-1)\lim_{t\to 1^+}h(t)=\lim_{t\to 1^+}\frac{\max\{\ln(g(t)),-\ln(g(t))\}}{\ln(t)}=\Big|\lim_{t\to 1^+}\frac{\ln(g(t))}{\ln(t)}\Big|,$$
where the last equality follows by continuity. As $\ln(1)=\ln(g(1))=0$, L'Hopital's rule implies that 
$$\lim_{t\to 1^+}\frac{\ln(g(t))}{\ln(t)}=\lim_{t\to 1^+}\frac{t \,g'(t)}{g(t)}=\lim_{t\to 1^+}-\frac{(p-1) t \left(\epsilon ^2-1\right)^2 (t+\epsilon )^p (t \epsilon +1)^p}{\zeta_1(t) \zeta_2(t)}$$
where 
$$\zeta_1(t)=\left(t \epsilon ^2 (t+\epsilon )^p+t (t \epsilon +1)^p+\epsilon  \left((t+\epsilon )^p+(t \epsilon +1)^p\right)\right)$$
and
$$\zeta_2(t)=\left(\epsilon ^2 (t \epsilon +1)^p+(t+\epsilon )^p+t \epsilon  \left((t+\epsilon )^p+(t \epsilon +1)^p\right)\right).$$
As $\zeta_1(1)\zeta_2(1)=(1+\epsilon)^{2p}(1+\epsilon)^4$, after rearrangement, we finally obtain
$$\lim_{t\to 1^+}h(t)=\Big|\frac{(p-1) \left(\epsilon ^2-1\right)^2 (1+\epsilon )^{2p}}{(q-1)\zeta_1(1) \zeta_2(1)}\Big|=\tau(\SS_{A_{\epsilon}}),$$
which implies $\tau(\SS_{A_{\epsilon}})\leq \kappa_H(\SS_{A_{\epsilon}})$ and thus concludes the proof.
\end{proof}
Now, we prove that the nonnegative critical points of $f_{A_\epsilon}$ are positive and we then characterize them in terms of a real parameter $t$. As critical points are defined up to multiples, we restrict our attention to the line $\{x\in\RR^2:x_1+x_2=1\}$. 

\begin{lemma}\label{eedef}
Let $x\in\RR^2_+$ with $x_1+x_2=1$. Then $x$ is a critical point of $f_{A_\epsilon}$ if and only if there exists $t\in (0,1)$ such that $x=(t,1-t)^\top$ and $\psi(t)=\psi(1-t)$ where $\psi\colon[0,1]\to\RR_{+}$ is defined as
\begin{equation}\label{eq:psi}
\psi(t)=t^{q-1} \big[(t\epsilon+1-t)^{p-1}+\epsilon (\epsilon+t-t\epsilon)^{p-1}\big].
\end{equation}
\end{lemma}
\begin{proof}
As we already observed, $f_{A_\epsilon}$ attains a global maximum in $\RR^2_{+}$. Furthermore, the critical points of $f_{A_\epsilon}$ satisfy
\begin{equation}\label{eeq1}
A_\epsilon \Phi_p( A_\epsilon x)=\lambda \Phi_q(x)\qquad x\in\RR^2\setminus \{0\}.
\end{equation}
%where, as already mentioned in the introduction, for any $p>1$, $\Phi_p$ is defined entry-wise as $\Phi_p(x)_i=|x_i|^{p-2}x_i$. 
As $A_{\epsilon}$ is positive, \eqref{eeq1} implies that every nonnegative critical point of $f_{A_\epsilon}$ is positive. It follows that, for positive vectors $x$, \eqref{eeq1} is equivalent to
\begin{equation}\label{eeq2}
\begin{cases}
\big(A_\epsilon \Phi_p( A_\epsilon x)\big)_1\, x_2^{q-1}=\big(A_\epsilon \Phi_p( A_\epsilon x)\big)_2\, x_1^{q-1} & \\ 
\lambda =(A_\epsilon \Phi_p( A_\epsilon x))_1 / x_1^{q-1} & 
\end{cases}%,\qquad x\in\RR^2_{++}.
\end{equation}
Thus, $x_1+x_2=1$ and $x_1,x_2>0$ imply the existence of $t\in(0,1)$ such that $x_1=t$ and $x_2=1-t$. Substituting $x=(t,1-t)^\top$ in \eqref{eeq2} we finally obtain the claimed result.
\end{proof}
A direct consequence of Lemma \ref{eedef} is that $(1,1)^{\top}/2$ is a critical point of $f_{A_\epsilon}$. Moreover, by symmetry, we see that $(t,1-t)^\top$ is a critical point of $f_{A_\epsilon}$ if and only if $(1-t,t)^\top$ is also a critical point. This observation implies the following 
\begin{lemma}\label{lem:tau_larger_than_1}
If $\tau(\SS_{A_{\epsilon}})>1$, then
$f_{A_\epsilon}$ has at least three distinct positive critical points. % on the sphere.
\end{lemma}
\begin{proof}
Note that if $\tau(\SS_{A_{\epsilon}})>1$, then $
\big(\frac{1+\epsilon}{1-\epsilon}\big)^2<\frac{p-1}{q-1}$.
Let $h\colon[0,1]\to \RR$ be defined as $h(t)=\psi(1-t)-\psi(t)$, where $\psi$ is defined as in \eqref{eq:psi}. The critical points of $f_{A_\epsilon}$ correspond to zeros of $h$ in $(0,1/2]$. Indeed, by Lemma \ref{eedef}, we know that these points are in bijection with the zeros of $h$ on $(0,1)$ and $h(t)=-h(1-t)$ for every $t\in(0,1)$. We have already observed that $h(t_0)=0$ with $t_0=1/2$. We now show that there exists $t_1\in(0,t_0)$ such that $h(t_1)=0$. The existence of such $t_1$ implies that $(t_1,1-t_1)^\top,(1-t_1,t_1)^\top,(t_0,t_0)^{\top}$ are three distinct positive critical points of $f_{A_\epsilon}$, since $h(1-t_1)=h(t_1)=0$.
To construct $t_1$, we first prove that our assumption $\tau(\SS_{A_{\epsilon}})>1$ is equivalent to the condition $h'(t_0)>0$.
We have 
\begin{align*}
\psi'(t)&=(q-1)t^{q-2}\big[(t\epsilon+1-t)^{p-1}+\epsilon (\epsilon+t-t\epsilon)^{p-1}\big]\\ &+(p-1)t^{q-1}(1-\epsilon)\big[\epsilon (\epsilon+t-t\epsilon)^{p-2}-(t\epsilon+1-t)^{p-2}\big].
\end{align*}
With $(\epsilon+t_0-t_0\epsilon)=(t_0\epsilon+1-t_0)=(\epsilon+1)/2$ we get
\begin{align*}
\psi'(t_0)& =(q-1)2^{2-q}(1+\epsilon)\Big(\frac{\epsilon+1}{2}\Big)^{p-1} +(p-1)2^{1-q}(1-\epsilon)(\epsilon-1)\Big(\frac{\epsilon+1}{2}\Big)^{p-2}\\
&= 2^{3-q-p}(1+\epsilon)^{p-2}\Big[(q-1)(1+\epsilon)^2-(p-1)(1-\epsilon)^2\Big].
\end{align*}
As $h'(t_0)=-\psi'(t_0)-\psi'(1-t_0)=-2\psi'(t_0)$,  we have $h'(t_0)>0$ if and only if $(q-1)(1+\epsilon)^2<(p-1)(1-\epsilon)^2$ i.e.\ $h'(t_0)>0$ if and only if $\tau(\SS_{A_{\epsilon}})>1$.

Now, as $h'(t_0)>0$, there exists a neighborhood $U$ of $t_0$ such that $h$ is strictly increasing on $U$. Since $h(t_0)=0$, this implies that there exists $s\in (0,t_0)\cap U$ such that $h(s)<0$.
As $\lim_{t\to 0}h(t)=\epsilon^{p-1}+\epsilon >0$, the intermediate value theorem implies the existence of $t_1\in(0,s)$ such that $h(t_1)=0$. As observed above, this concludes the proof.
\end{proof}

Finally, we address the case $\tau(\SS_{A_\epsilon})=1$. % \begin{lem}

\begin{lemma}\label{lem:tau_equals_to_1}
If $\tau(\SS_{A_\epsilon})= 1$, then $f_{A_\epsilon}$ has a unique nonnegative critical point. % of the form $(s,s)$.
\end{lemma}
\begin{proof}
Let $F\colon\RR^2_+\to \RR_+^2$ be defined as
$F(x)=\Phi_{q^*}(A_{\epsilon}\Phi_p(A_{\epsilon}x))$, where $q^*=q/(q-1)$ denotes the H\"older conjugate of $q$.  Then, for $\ones=(1,1)^\top$ and $u=\ones/2$, we have $F(u)=\lambda u$ for some $\lambda>0$. Hence, $u$ is a fixed point of $\SS_{A_{\epsilon}}$ and, $\|\cdot\|_q$ is differentiable, by Lemma \ref{lem:crazy}, it follows that $u$ is a critical point of $f_{A_\epsilon}$. Moreover, it is a fixed point of $G\colon D_+\to D_+$ defined by $G(x)=\langle{F(x),\ones\rangle}^{-1}F(x)$, where %$e$ is the vector of all ones and 
$D_+=\{(t,1-t): t\in[0,1]\}$. Note that the fixed points of $G$ coincide, up to scaling, with those of $\SS_{A_{\epsilon}}$. To conclude, we prove that $u$ is the unique fixed point of $G$.

As $\tau(\SS_{A_\epsilon})=1$, we have $d_H(G(x),G(y))=d_H(F(x),F(y))\leq d_H(x,y)$ and so $G$ is non-expansive with respect to $d_H$. Now, Theorem 6.4.1 in \cite{lemmens2012nonlinear} implies that $u$ is the unique fixed point of $G$, if 
\begin{equation*}
z-G'(u)z\neq 0\qquad \forall z\in\RR^2\setminus \{0\} \text{ with }z_1+z_2=0.
\end{equation*}
where $G'(u)$ denotes the Jacobian matrix of $G$ evaluated at $u$. 
Moreover, as $F(u)=\lambda u$, Lemma 6.4.2 in \cite{lemmens2012nonlinear} implies that $F'(u)u=\lambda u$  and %with $\rho$ being the spectral radius of $F'(u)$, and
\begin{equation*}
G'(u)z = \tfrac{1}{\lambda}(F'(u)z-\langle{F'(u)z,\ones}\rangle u).
\end{equation*}

Suppose by contradiction that there exists a $z\in\RR^2\setminus \{0\}$ with $z_1+z_2=0$, such that $z-G'(u)z=0$. A direct computation shows that $\langle{z,F'(u)^T u\rangle}=0$. Then, 
\begin{align*}
0&= z-G'(u)z = z-\tfrac{1}{\lambda}F'(u)z+\tfrac{1}{\lambda}\langle{F'(u)z,\ones}\rangle u \\&=z-\tfrac{1}{\lambda}F'(u)z+\tfrac{2}{\lambda}\langle{z,F'(u)^\top u}\rangle u=z-\tfrac{1}{\lambda}F'(u)z.
\end{align*}
It follows that $F'(u)z=\lambda z$ and, as $F'(u)$ is entry-wise positive, the classical Perron-Frobenius theorem implies that $z=\pm u$. However, $u_1+u_2>0$ which contradicts the assumption $z_1+z_2=0$. So $0 \neq z-G'(u)z$ for every $z\neq 0$ such that $z_1+z_2=0$. Hence, $u$ is the unique fixed point of $G$, which concludes the proof.
\end{proof}

Combining the last two lemmas allows us to conclude:
\begin{proof}[Proof of Theorem \ref{thm:thightness}]
Due to Lemmas \ref{lem:tau_larger_than_1} and \ref{lem:tau_equals_to_1} we only need to address the case $\tau(S_{A_{\epsilon}})< 1$. However this is a direct consequence of Lemma \ref{PFmono}. In fact, as $A_{\epsilon}$ is entry-wise positive, the nonnegative fixed points of $\SS_{A_\epsilon}$ are positive and, if $\tau(S_{A_{\epsilon}})<1$, then $\SS_{A_{\epsilon}}$ is a strict contraction with respect to $d_H$ and so it has a unique  fixed point which also is the unique positive maximizer of $f_{A_\epsilon}$ on $\RR^2_+$. 
\end{proof}

%%%%%%%%%%%%%%%%%%%%%%%%%%%%%%%%%%%%%%%%
%%%% SYMBOL FOR SETS of INDICES %%%%%%%%
\def\I{\mathcal V}
%%%%%%%%%%%%%%%%%%%%%%%%%%%%%%%%%%%%%%%%

\section{Matrix norms induced by sum of weighted $\ell^p$ norms}\label{sec:sum_p_norms}
The Birkhoff contraction ratios  $\kappa_H(J_\alpha)$ and $\kappa_H(J_{\alpha^*})$ are easy to compute when $\|\cdot\|_\alpha$ is a weighted $\ell^p$ norm. More precisely, we have the following \begin{proposition}\label{dualweightedp}
Let $\norm{x}_{\alpha}=\norm{Dx}_p$ for some $p\in(1,\infty)$ and some diagonal matrix $D$ with positive diagonal entries, then $\norm{x}_{\alpha^*}=\norm{D^{-1}x}_{p*}$ where $p^*=p/(p-1)$. Furthermore, it holds $\kappa_H(J_{\alpha})=\kappa_H(J_{\alpha^*})^{-1}=p-1$.
\end{proposition}
\begin{proof}
The equality $\norm{x}_{\alpha^*}=\norm{D^{-1}x}_{p*}$ follows from Theorem \ref{dualthm} below. To conclude, note that $J_{\alpha}(x)=\norm{Dx}_{p}^{1-p}D^p\Phi_p(x)$ and therefore $\kappa_H(J_{\alpha})=\kappa_H(\Phi_p)=p-1$. The same argument shows that $\kappa_H(J_{\alpha^*})=\kappa_H(\Phi_{p^*})=p^*-1=(p-1)^{-1}$.
\end{proof}

%a straightforward computation shows that $M(J_{\alpha}(x)/J_{\alpha}(y))=M(x/y)^{p-1}$ for every $x,y\in\RR^n_+\setminus\{0\}$ with $x\sim y$. It is then clear that $\kappa_H(J_{\alpha})=\kappa_H(J_{\alpha^*})=p-1$. 
While the above Proposition \ref{dualweightedp} makes the computation of the Birkhoff constant of weighted $\ell^p$-norms particularly easy, computing $\kappa_H(J_{\alpha})$ or $\kappa_H(J_{\alpha^*})$ for a general strongly monotonic norm $\norm{\cdot}_{\alpha}$ can be a difficult task. There are norms for which an explicit expression in terms of arithmetic operations for $\norm{\cdot}_{\alpha}$ is given by construction (resp. modelisation), but such an expression is not available for the dual $\norm{\cdot}_{\alpha_*}$. Examples include $\norm{x}_{\alpha}=(\norm{x}^3_{p}+\norm{x}^3_{q})^{1/3}$ as shown by Theorem \ref{dualthm} below. On the other hand, as discussed in the introduction, monotonic norms different than the standard $\ell^p$ norms arise quite naturally in several applications. % \cite{duchenne2011tensor, khot2012grothendieck,xue2000efficient}. 

Motivated by the above observations, we devote the rest of the section to the study of a particular class of monotonic norms
of the form 
$\norm{x}_{\alpha}=\norm{\big(\norm{x}_{\alpha_1},\ldots,\norm{x}_{\alpha_d}\big)}_{\gamma}$ where all the norms are monotonic and where we also allow $\norm{x}_{\alpha_i}$ to measure only a subset of the coordinates of $x$.

\subsection{Composition of monotonic norms and its dual}\label{dual_section}
Let $d$ be a positive integer. We consider norms of the following form 
\begin{equation}\label{normeq}
\norm{x}_{\alpha}=\norm{\big(\norm{P_1x}_{\alpha_1},\ldots,\norm{P_dx}_{\alpha_d}\big)}_{\gamma} 
\end{equation}
where $\norm{\cdot}_{\gamma}$ is a monotonic norm on $\RR^d$, $\norm{\cdot}_{\alpha_i}$ is a norm on $\RR^{n_i}$ and $P_i\in\RR^{n_i\times n}$ is a ``weight matrix'' for all $i=1,\ldots,d$. For $\norm{\cdot}_{\alpha}$ to be a norm, we assume that $M= [P_1^\top,\ldots,P_d^{\top}]^\top\in\RR^{ (n_1+\ldots+n_d)\times n}$ has rank $n$. 
Note that the monotonicity of $\norm{\cdot}_{\gamma}$ implies that $\norm{\cdot}_{\alpha}$ satisfies the triangle inequality. %If $P_i\in\RR^{n_i\times n}_+$ and $\norm{\cdot}_{\alpha_i}$ is a monotonic norm on $\RR^{n_i}$ for all $i$, then,

Let us first discuss particular cases of \eqref{normeq}. First, note that for two norms $\norm{\cdot}_{\alpha_1},\norm{\cdot}_{\alpha_2}$ on $\RR^n$, the norm $$\norm{x}_{\alpha_+}=(\norm{x}_{\alpha_1}^p+\norm{x}_{\alpha_2}^p)^{1/p}$$ 
can be obtained from \eqref{normeq} with $d=2$, $\norm{\cdot}_{\gamma}=\norm{\cdot}_p$, and $P_1=P_2 = I$, with $I\in\RR^{n\times n}$ being the identity matrix. It is also possible to model norms acting on different coordinates of the vectors. For example, if $(x,y)\in\RR^{2n}$, then 
$$\norm{(x,y)}_{\alpha_{\times}}=(\norm{x}_{\alpha_1}^p+\norm{y}_{\alpha_2}^p)^{1/p}$$
can be obtained from \eqref{normeq} with $d=2$,  $\norm{\cdot}_{\gamma}=\norm{\cdot}_p$, $P_1=\diag(1,\ldots,1,0,\ldots,0)\in\RR^{2n\times 2n}$ and $P_2=\diag(0,\ldots,0,1,\ldots,1)\in\RR^{2n\times 2n}$. The dual of $\norm{\cdot}_{\alpha_{\times}}$ is discussed in Lemma \ref{45ds254asd354} below and has a particularly elegant description. More complicated weight matrices $P_i$ can also be used. For example if $\tilde n$ is an integer not smaller than $n$ and $P\in\RR^{\tilde n\times n}$ has rank $n$, then the norm $$\norm{x}_{\alpha_P} = \norm{Px}_{p}$$
can be obtained with $d=1$, $\norm{\cdot}_{\gamma}=|\cdot|$, $\norm{\cdot}_{\alpha_1}=\norm{\cdot}_p$ and $P_1 = P$. 
Note that if $\tilde n = n$, then $P$ is square and invertible and this property can be used to simplify the evaluation of 
the dual norm of $\norm{\cdot}_{\alpha_P}$.
Consequences of such additional structure 
are discussed in Corollary \ref{cor:dual_sum_of_norms_empty_intersection}.

In the next Theorem \ref{dualthm} we provide a characterization of the dual norm of $\norm{\cdot}_{\alpha}$ in its general form as defined in \eqref{normeq}. We first need the following lemma that addresses the particular case where $P_1,\ldots,P_d$ are projections.

%We first prove the following particular case which can be obtained by setting $P_1,\ldots,P_d$ to be projections in Theorem \ref{dualthm}.
\begin{lemma}\label{45ds254asd354}
	Let $n_1,\ldots,n_d$ be positive integers and for $i=1,\ldots,d$ let $\norm{\cdot}_i$ be a norm on $\RR^{n_i}$. Furthermore, let $\norm{\cdot}_{\gamma}$ be a monotonic norm on $\RR^d$. Let $V= \RR^{n_1}\times \ldots \times \RR^{n_d}$ and for all $(u_1,\ldots,u_d)\in V$ define
	$$\norm{(u_1,\ldots,u_d)}_{V}=\norm{\big(\norm{u_1}_{\alpha_1},\ldots,\norm{u_d}_{\alpha_d}\big)}_{\gamma}\, .$$
	Then $\norm{\cdot}_{ V}$ is a norm on $ V$ and the induced dual norm $\norm{\cdot}_{ V^*}$ satisfies
	$$\norm{(u_1,\ldots,u_d)}_{ V^*}=\norm{\big(\norm{u_1}_{\alpha_1^*},\ldots,\norm{u_d}_{\alpha_d^*}\big)}_{\gamma^*}\qquad \forall (u_1,\ldots,u_d)\in V\, .$$
	%where we identify $ V^*$ with $ V$.
\end{lemma}
\begin{proof}
	The fact that $\norm{\cdot}_{ V}$ is a norm follows from a direct verification. Let $(u_1,\ldots,u_d)\in  V$. Then, for every $(y_1,\ldots,y_d)\in V$, we have
	\begin{align*}
	\ps{(u_1,\ldots,u_d)}{(y_1,\ldots,y_d)}&=\sum_{i=1}^d \ps{u_i}{y_i}\leq \sum_{i=1}^d \norm{u_i}_{\alpha_i^*}\norm{y_i}_{\alpha_i} \\ &\leq \norm{\big(\norm{u_1}_{\alpha_1^*},\ldots,\norm{u_d}_{\alpha_d^*}\big)}_{\gamma^*}\norm{\big(\norm{y_1}_{\alpha_1},\ldots,\norm{y_d}_{\alpha_d}\big)}_{\gamma},
	\end{align*}
	which shows that
	\begin{equation}\label{sad354354}
	\norm{(u_1,\ldots,u_d)}_{ V^*}\leq\norm{\big(\norm{u_1}_{\alpha_1^*},\ldots,\norm{u_d}_{\alpha_d^*}\big)}_{\gamma*}.
	\end{equation}
	For the reverse inequality, let $v = (\norm{u_1}_{\alpha_1^*},\ldots,\norm{u_d}_{\alpha_d^*})$. As $\|\cdot\|_\gamma$ is monotonic, by Proposition 5.2 in \cite[Chapter 1]{cioranescu2012geometry}, there exists $w\in \RR^d_+$ such that $\norm{w}_{\gamma}\leq 1$ and $\ps{v}{w}=\norm{v}_{\gamma^*}$. Let us denote by $w_1,\ldots,w_d\in\RR_+$ and $v_1,\ldots,v_d\in\RR$ respectively the components of $w$ and $v$ in the canonical basis of $\RR^d$. Now, let  $ \bar y_1\in\RR^{n_1},\ldots,\bar y_d\in\RR^{n_d}$ be such that $\norm{\bar y_i}_{\alpha_i}\leq 1$ and $\ps{\bar y_i}{u_i}=\norm{u_i}_{\alpha_i^*}$ for all $i=1,\ldots,d$. Then, as $\norm{\cdot}_{\gamma}$ is monotonic with respect to $\RR^d_+$ and $\|\bar y_i\|_{\alpha_i}\leq 1$ for all $i$, we have 
	$$\norm{\big(\norm{w_1\,\bar y_1}_{\alpha_1},\ldots,\norm{w_d\,\bar y_d}_{\alpha_d}\big)}_{\gamma}= \norm{\big(w_1\norm{\bar y_1}_{\alpha_1},\ldots,w_d\norm{\bar y_d}_{\alpha_d}\big)}_{\gamma}\leq\norm{w}_{\gamma}\leq 1.$$
	Note that 
	\begin{align*}
	\ps{(u_1,\ldots,u_d)}{(w_1\,\bar y_1,\ldots, w_d\,\bar y_d)}&=\sum_{i=1}^d w_i\ps{u_i}{\bar y_i}=\sum_{i=1}^d w_i\,\norm{u_i}_{\alpha_i^*}=\ps{v}{w}\\ &=\norm{v}_{\gamma^*}=\norm{\big(\norm{u_1}_{\alpha_1^*},\ldots,\norm{u_d}_{\alpha_d^*}\big)}_{\gamma*}.
	\end{align*}
	It follows that 
	$\norm{\big(\norm{u_1}_{\alpha_1^*},\ldots,\norm{u_d}_{\alpha_d^*}\big)}_{\gamma*}\leq \norm{(u_1,\ldots,u_d)}_{ V^*},$
	which, together with \eqref{sad354354}, concludes the proof. 
\end{proof}
	
\begin{theorem}\label{dualthm}
	Let $d$ be a positive integer. For $i=1,\ldots,d$, let $P_i\in\RR^{n_i\times n}$ and let $\norm{\cdot}_{\alpha_i}$ be a norm on $\RR^{n_i}$. Suppose that $M= [P_1^\top,\ldots,P_d^{\top}]^\top\in\RR^{ (n_1+\ldots+n_d)\times n}$ has rank $n$.
	Furthermore, let $\norm{\cdot}_{\gamma}$ be a monotonic norm on $\RR^d$. For every $x\in \RR^n$, define
	$$\norm{x}_{\alpha}=\norm{\big(\norm{P_1x}_{\alpha_1},\ldots,\norm{P_d x}_{\alpha_d}\big)}_{\gamma}.$$
	Then, $\norm{\cdot}_{\alpha}$ is a norm on $\RR^n$ and the induced dual norm is given by 
	$$\norm{x}_{\alpha^*}= \inf_{\substack{u_1\in\RR^{n_1},\ldots,u_d\in\RR^{n_d}\\ P_1^\top u_1+\ldots +P_d^\top u_d = x}} \norm{\big(\norm{u_1}_{\alpha_1^*},\ldots,\norm{u_d}_{\alpha_d^*}\big)}_{\gamma^*},$$
	where $\norm{\cdot}_{\alpha_i^*}$ is the dual norm induced by $\norm{\cdot}_{\alpha_i}$ and $\norm{\cdot}_{\gamma^*}$ is the dual norm induced by $\norm{\cdot}_{\gamma}$. 
\end{theorem}
\begin{proof}%[Proof of Theorem \ref{dualthm}]
	Let $u_1\in\RR^{n_1},\ldots,u_d\in\RR^{n_d}$ be such that $P_1^\top u_1+\ldots +P_d^\top u_d=x$. Such vectors always exists as $M$ has full rank. Then, for every $y\in\RR^n$, it holds
	\begin{align*}
	\ps{x}{y}& = \sum_{i=1}^d \ps{P_i^\top u_i}{y}= \sum_{i=1}^d \ps{u_i}{P_iy}\\ &\leq \sum_{i=1}^d \norm{u_i}_{\alpha_i^*}\norm{P_iy}_{\alpha_i}\leq \norm{\big(\norm{u_1}_{\alpha_1^*},\ldots,\norm{u_d}_{\alpha_d^*}\big)}_{\gamma^*}\norm{y}_{\alpha}.
	\end{align*}
	It follows that 
	$$\norm{x}_{\alpha^*}\leq \inf_{\substack{u_1\in\RR^{n_1},\ldots,u_d\in\RR^{n_d}\\ P_1^\top u_1+\ldots +P_d^\top u_d = x}} \norm{\big(\norm{u_1}_{\alpha_1^*},\ldots,\norm{u_d}_{\alpha_d^*}\big)}_{\gamma^*}.$$
	Now, we prove the reverse inequality. To this end, consider the vector space $ V=\RR^{n_1}\times \ldots \times \RR^{n_d}$ endowed with the norm $\norm{\cdot}_{ V}$ defined as 
	$$\norm{(u_1,\ldots,u_d)}_{ V}=\norm{\big(\norm{u_1}_{\alpha_1},\ldots,\norm{u_d}_{\alpha_d}\big)}_{\gamma}\qquad \forall (u_1,\ldots,u_d)\in V.$$
%	Note that for all $(u_1,\ldots,u_m)\in  V$, and $(y_1,\ldots,y_m)\in  V$ it holds
%	\begin{align}
%	\sup_{\norm{(y_1,\ldots,y_m)}_{ V}\leq 1}\ps{(u_1,\ldots,u_m)}{(y_1,\ldots,y_m)}&=\sup_{\norm{(y_1,\ldots,y_m)}_{ V}\leq 1}\sum_{i=1}^m\ps{u_i}{y_i}
%	\end{align}
	As $ V$ is a finite product of finite dimensional vector spaces, we can identify $ V^*$ with $ V$ and
	by Lemma \ref{45ds254asd354}, we know that the dual norm of $\norm{\cdot}_{ V^*}$ induced by $\norm{\cdot}_{ V}$ satisfies
	$$\norm{(u_1,\ldots,u_d)}_{ V^*}=\norm{\big(\norm{u_1}_{\alpha_1^*},\ldots,\norm{u_d}_{\alpha_d^*}\big)}_{\gamma^*}\qquad \forall (u_1,\ldots,u_d)\in V.$$
	%To prove this claim,
	Consider now the vector subspace $ W=\{(P_1y,\ldots,P_dy)\mid y\in\RR^n\}\subset V$. Note that, we can identify $ W$ with the image of $M$, i.e. $ W = \{My\mid y\in\RR^n\}$. Let $M^\dagger\in\RR^{n\times (n_1+\ldots+n_d)}$ be the Moore-Penrose inverse of $M$. Then, as $M$ is full rank, we have $M^\dagger My= y$ for all $y\in\RR^n$. Let $\phi\colon  W\to \RR$ be defined as 
	$$\phi(u_1,\ldots,u_d)=\ps{M^\dagger(u_1,\ldots,u_d)}{x}\qquad \forall (u_1,\ldots, u_d)\in  W.$$
	For every $ (u_1,\ldots, u_d)\in  W$, there exists $y\in\RR^n$ such that $(u_1,\ldots, u_d)=My$, i.e. $u_i=P_iy$ for all $i=1,\ldots,d$, and thus
	\begin{align*}
	|\phi(u_1,&\ldots,u_d)|=|\phi(My)|=|\ps{M^\dagger My}{x}|=|\ps{y}{x}|\\ &\leq  \norm{y}_{\alpha}\norm{x}_{\alpha^*}=\norm{\big(\norm{P_1y}_{\alpha_1},\ldots,\norm{P_dy}_{\alpha_d}\big)}_{\gamma}\norm{x}_{\alpha^*}=\norm{(u_1,\ldots,u_d)}_{ V}\norm{x}_{\alpha^*}.
	\end{align*}
	 By the Hahn--Banach %extension 
	 theorem (see e.g.\ Corollary 1.2 of \cite{brezis2010functional}), there exists $(u_1',\ldots,u_d')\in  V$ such that \begin{equation}\label{dsfkjhdskfh4254}
	 \phi(u_1,\ldots,u_d)=\sum_{i=1}^d \ps{u_i'}{u_i}\qquad \forall (u_1,\ldots,u_d)\in  W,
	 \end{equation}
	 and $$\norm{\big(\norm{u'_1}_{\alpha_1^*},\ldots,\norm{u'_d}_{\alpha_d^*}\big)}_{\gamma^*}=\norm{(u'_1,\ldots,u'_d)}_{V^*}\leq \norm{x}_{\alpha^*}.$$
	 Next, let $y\in\RR^n$, then $My=(P_1y,\ldots,P_dy)\in  W$ and with \eqref{dsfkjhdskfh4254}, we have
	 $$
	 \ps{y}{x}=\ps{M^\dagger My}{x}=\sum_{i=1}^d \ps{u_i'}{P_iy}=\sum_{i=1}^d \ps{P_i^\top u_i'}{y}.
	 $$
	 As the above is true for all $y\in \RR^n$, it follows that $P_1^\top u_1'+\ldots +P_d^\top u_d'=x$. Hence, we have
	  $$\inf_{\substack{u_1\in\RR^{n_1},\ldots,u_d\in\RR^{n_d}\\ P_1^\top u_1+\ldots +P_d^\top u_d = x}} \norm{\big(\norm{u_1}_{\alpha_1^*},\ldots,\norm{u_d}_{\alpha_d^*}\big)}_{\gamma^*}\leq\norm{\big(\norm{u'_1}_{\alpha_1^*},\ldots,\norm{u'_d}_{\alpha_d^*}\big)}_{\gamma^*}\leq \norm{x}_{\alpha^*},$$
	  which concludes the proof of the formula for $\norm{\cdot}_{\alpha^*}$.
	  \end{proof}

As a consequence of the above Theorem \ref{dualthm}, we have that the dual of the norms $\norm{\cdot}_{\alpha_+},\norm{\cdot}_{\alpha_{\times}},\norm{\cdot}_{\alpha_P}$ considered at the beginning of this section are respectively given by
$$\norm{x}_{\alpha_+^*}=\inf_{\substack{u_1+u_2=x\\ u_1,u_2\in\RR}}(\norm{u_1}_{\alpha_1^*}^{p^*}+\norm{u_2}_{\alpha_2^*}^{p^*})^{1/p^*},$$ $$\norm{(x,y)}_{\alpha_{\times}^*}=(\norm{x}_{\alpha_1^*}^{p^*}+\norm{y}_{\alpha_2^*}^{p^*})^{1/p^*},\qquad \norm{x}_{\alpha_P^*}=\inf_{u \in \RR^{\tilde n}\colon P^\top u = x}\norm{u}_{p^*},$$
with $p^* = p/(p-1)$. Note that the $\norm{\cdot}_{\alpha^*_{\times}}$ does not involve an infimum. The infimum can also be removed in $\norm{x}_{\alpha_P^*}$, if $P$ is square and invertible and in that case it holds $\norm{x}_{\alpha_P^*}=\norm{P^{-\top}x}_{p^*}$. 

We discuss more general examples in the next result.
\begin{corollary}\label{cor:dual_sum_of_norms_empty_intersection} 
Under the same assumptions as Theorem \ref{dualthm}, we have:
\begin{enumerate}
\item If $P_1,\ldots,P_d$ are all square invertible matrices and 
$$
\norm{x}_{\alpha^*}=\min_{\substack{x=u_1+\cdots + u_d\\ u_1,\ldots,u_d\in\RR^n}}\norm{(\norm{(P_1^\top)^{-1} u_1}_{\alpha_1^*},\ldots,\norm{(P_d^\top)^{-1} u_d}_{\alpha_d^*})}_{\gamma^*}
$$

\item If  every $x\in\RR^n$ can be uniquely written as $x=x_{P_1}+\ldots+x_{P_d}$ with $x_{P_i}\in\operatorname{Im}(P_i^\top)$ for all $i=1,\ldots,d$ (i.e.\ $\RR^n$ is the direct sum of the range of $P_1,\dots,P_d$), then 
$$\norm{x}_{\alpha^*}=\norm{\Big(\inf_{\substack{u_1\in\RR^{n_1}\\ P_1^\top u_1= x_{P_1}}}\norm{u_1}_{\alpha_1^*},\ldots,\inf_{\substack{u_d\in\RR^{n_d}\\ P_d^\top u_d= x_{P_d}}}\norm{u_{d}}_{\alpha_d^*}\Big)}_{\gamma^*}.$$ 
If, additionally, $n_i = \dim(\operatorname{Im}(P_i^\top))$ for all $i=1,\ldots,d$, then
$$\norm{x}_{\alpha^*}=\norm{\big(\norm{(P_1^\top)^{\dagger}x}_{\alpha_1^*},\ldots,\norm{(P_d^\top)^{\dagger} x}_{\alpha_d^*}\big)}_{\gamma^*},$$
where $(P_i^\top)^{\dagger}$ is the Moore-Penrose inverse of $P_i^\top$.
\end{enumerate}
\end{corollary}

\subsection{The power method for compositions of $\ell^p$-norms}\label{sumPF_section}

We discuss here consequences of Theorems  \ref{newPF} and \ref{dualthm} when applied to a special family of norms
defined in terms of subsets of entries of the initial vector, i.e.\ the case where $P_i$ is a nonnegative diagonal matrix.

For some nonnegative weight vector $\omega\in\RR^n$ and coefficient $p\in (1,\infty)$, let $\norm{\cdot}_{\omega,p}$ be the $\omega$-weighted $\ell^p$-(semi)norm on $\RR^n$, defined as 
\begin{equation}\norm{x}_{\omega,p}=\norm{\diag(\omega)^{1/p}x}_p=\Big(\sum_{k=1}^n\omega_i|x_i|^p\Big)^{1/p}.
\end{equation}
To express the dual of $\norm{x}_{\omega,p}$ and their compositions, let
\begin{equation}\label{dualweightpnorm}
p^*=\frac{p}{p-1}\quad\text{and}\quad\omega_i^*=\begin{cases} \omega_i^{1-p^*}&\text{if }\omega_i>0,\\
0&\text{if }\omega_i=0,
\end{cases}
\quad\forall i=1,\ldots,n.
\end{equation}
If $\omega$ is positive, then $\norm{x}_{\omega,p}$ is a norm and it holds $(\norm{x}_{\omega,p})_*=\norm{x}_{\omega^*,p^*} $ by Proposition \ref{dualweightedp}. 

Let $\omega_1,\ldots,\omega_d\in\RR^m$ %and $\varpi_1,\ldots,\varpi_h\in\RR^n$ 
be nonzero vectors of nonnegative weights such that  $\omega_1+\ldots+\omega_d$ %and $\varpi_1+\ldots+\varpi_h$ 
is a positive vector. Further let $s\in[1,\infty)$, $p_1,\ldots,p_d\in(1,\infty)$ %and $q_1,\ldots,q_h\in(1,\infty)$, 
and define \begin{equation}\label{eq:sum_of_p_norm}
\norm{x}_{\alpha}
=\Big(\sum_{k=1}^d \norm{x}_{\omega_k,p_k}^s\Big)^{1/s},
%\qquad 
%\norm{x}_{\beta}=\Big(\sum_{k=1}^h \norm{x}_{\varpi_k,q_k}^t\Big)^{1/t} \, .
\end{equation}
The fact that $\omega_1+\dots+\omega_d$ is positive ensures that $\norm{\cdot}_{\alpha}$ is a norm. 
Note that $\norm{\cdot}_{\alpha}$  is strongly monotonic. The differentiability of $\norm{\cdot}_{\alpha}$ is discussed in the following lemma. %, and an analogous result holds for $\norm{\cdot}_{\beta}$.
\begin{lemma}
Let $\norm{\cdot}_{\alpha}$ be as in \eqref{eq:sum_of_p_norm}, then $\norm{\cdot}_{\alpha}$ is differentiable if either $s>1$ or $s=1$ and $\omega_i$ has at least two positive entries for every $i=1,\ldots,d$.
\end{lemma}
\begin{proof}
As $p_k>1$, $\norm{\cdot}_{\omega_k,p_k}$ is differentiable if $\omega_{k}$ has at least two positive entries. If it has only one positive entry then $\norm{\cdot}_{\omega_k,p_k}$ is just a weighted absolute value. Hence, if $s>1$, then the differentiability of $\norm{\cdot}_{\alpha}$ follows from that of the $\ell^{s}$-norm. While if $s=1$ and $\omega_i$ has at least two positive entries for every $i=1,\ldots,d$, then $\norm{\cdot}_{\alpha}$ is just a sum of differentiable norms. \end{proof}
If $\|\cdot\|_\alpha$ is differentiable, we have %\mat{ to be consistent with how we have written $J_\alpha$ before I would write the first term as $\|x\|_{\alpha}^{1-s}$ ?}
\begin{equation}\label{eq:Jalpha_explicit}
J_\alpha(x) = \|x\|_{\alpha}^{1-s} 
%\Big(\sum_{k=1}^d \norm{x}_{\omega_k, p_k}^{s}\Big)^{(1-s)/s}
\sum_{k=1}^d\norm{x}_{\omega_k,p_k}^{s-p_k} \mathrm{diag}(\omega_k)\Phi_{p_k}(x)%,\quad \forall x \in\RR^m
\end{equation}
and the following  lemma provides an upper bound for $\kappa_H(J_{\alpha})$.
%The following Theorem \ref{sumPF} provides conditions for the computation of $\norm{A}_{\beta^*\to \alpha}$.  In order to prove the theorem we need one additional result which upper bounds $\kappa_H(J_{\alpha})$ and which we provide in the lemma below.
\begin{lemma}\label{lem:upperbound_pH}
Let $\norm{\cdot}_{\alpha}$ be as in \eqref{eq:sum_of_p_norm}. If $\norm{\cdot}_{\alpha}$ is differentiable then 
$$\kappa_H(J_{\alpha})\leq (s-1)+\sum_{k=1}^d\max\{0,p_k-s\}.$$
\end{lemma}
\begin{proof}
Let $\delta=\sum_{k=1}^d\max\{0,p_k-s\}$. We have
%$$J_{\alpha}(x) = \norm{x}_{\alpha}^{1-s}\sum_{k=1}^d\norm{x_{I_k}}_{p_k}^{s-p_k}\Phi_{p_k}(x_{I_k}).$$
%It follows that 
$J_{\alpha}(x)=\norm{x}_{\alpha}^{1-s}(F(x)+G(x))$ where for all $x\in\RR^m_{+}\saufzero$ we let 
$F(x) = \sum_{p_k\leq s}\norm{x}_{\omega_k,p_k}^{s-p_k}\diag(\omega_k)\Phi_{p_k}(x)$ and $ G(x)=\sum_{p_k>s}\norm{x}_{\omega_k,p_k}^{s-p_k}\diag(\omega_k)\Phi_{p_k}(x).$
Note that if $p_k > s$ for all $k$ then $F(x)=0$, whereas $G(x)=0$ when $p_k\leq s$ for all $k$. Moreover, note that $F$ is order-preserving and homogeneous of degree $s-1$. Now let us set $\tau(x)=1$ if $p_j\leq s$ for all $j$ and $\tau(x) =\prod_{p_j>s}\norm{x}_{\omega_j,p_j}^{p_j-s}$ otherwise. Then $\tau$ is order-preserving and homogeneous of degree $\delta$ and $x\mapsto\tau(x)F(x)$ is order-preserving and homogeneous of degree $\delta+(s-1)$. Finally, note that
$$x\mapsto\tau(x) G(x)=\sum_{p_k>s}\prod_{\substack{p_j>s\\ j \neq k}}\norm{x}_{\omega_j,p_j}^{p_j-s}\diag(\omega_k)\Phi_{p_k}(x)$$ is order-preserving as well and homogeneous of degree $\delta+(s-1)$. This implies that $\delta +(s-1)$ is a Lipschitz constant of $H(x)=\tau(x)(F(x)+G(x))$ with respect to the Hilbert metric $\mu$. Hence, for any $x,y\in \RR^m_+\saufzero$ with $x\sim y$, we finally obtain
\begin{equation*}
\mu(J_{\alpha}(x),J_{\alpha}(y))
%=\mu(F(x)+G(x),F(y)+G(y))
=\mu\big(H(x),H(y)\big)\leq (\delta +s-1)\mu(x,y),
\end{equation*}
which concludes the proof.
\end{proof}
If $s>1$, by Theorem \ref{dualthm}, we have
\begin{align}\label{eq:dd}
\norm{x}_{\alpha^*} &=\min_{\substack{u_1,\ldots,u_d\in\RR^n\\ \diag(\omega_1)u_1+\cdots+\diag(\omega_d)u_d=x}}\Big(\sum_{k=1}^d \norm{u_k}_{p_k^*}^{s^*}\Big)^{1/s^*} \notag\,\\ &=\,\min_{\substack{u_1+\cdots+u_d=x\\ u_1,\ldots,u_d\in\RR^n}}\Big(\sum_{k=1}^d \norm{u_k}_{\omega_k^*,p_k^*}^{s^*}\Big)^{1/s^*} \, .
\end{align}
%\mat{the variables $\varpi_k^*$ are not defined}
It is not difficult to realize that the case $s=1$ has a similar form, where the  sum  is replaced by a maximum. We henceforth omit that case, for the sake of brevity.

Now, consider a norm $\|\cdot\|_\beta$ defined as the dual norm of a norm of the type \eqref{eq:sum_of_p_norm}
\begin{equation}\label{eq:beta_norm_special}
    \|x\|_\beta = \min_{\substack{u_1+\cdots+u_h=x\\ u_1,\ldots,u_h\in\RR^n}}\Big(\sum_{k=1}^h \norm{u_k}_{\varpi_k,q_k}^{t}\Big)^{1/t} 
\end{equation}
where $h$ is some positive integer, $\varpi_i$ are nonnegative weight vectors whose sum $\varpi_1+\dots+\varpi_h$ is  positive  and $q_1,\dots,q_h, t\in (1,\infty)$. 
As $\min_x f(x)=(\max_x f(x)^{-1})^{-1}$ for continuous positive $f$, we deduce that for this choice of norm we have 
\begin{equation*}
\norm{A}_{\beta\to\alpha} =\max_{x\neq 0}\frac{\norm{Ax}_{\alpha}}{\norm{x}_{\beta}}=\max_{\substack{u_1+\cdots+u_h\neq 0\\ u_1\in\RR^n,\ldots,u_h\in\RR^n}}\dfrac{\displaystyle\Big(\sum_{k=1}^d \norm{\sum_{j=1}^h Au_j}_{\omega_k,p_k}^s\Big)^{1/s}}{\displaystyle\Big(\sum_{k=1}^h \norm{u_k}_{\varpi_k,q_k}^{t}\Big)^{1/t}}
\end{equation*}
for any matrix $A\in\RR^{m\times n}$.

We emphasize that, while the norm $\|\cdot \|_\beta$ is defined implicitly in the general case, when the weight vectors $\varpi_i$ have disjoint support, Corollary \ref{cor:dual_sum_of_norms_empty_intersection}  yields the following  explicit formula 
$$
\|x\|_\beta = \Big(\sum_{k=1}^h \norm{x}_{\varpi_k,q_k}^{t}\Big)^{1/t}  
$$
which also simplifies the definition of $\|A\|_{\beta\to \alpha}$. 

The advantage of choosing $\|\cdot \|_\beta$ as in \eqref{eq:beta_norm_special} relies on the fact that both $\|x\|_{\beta^*}$ and   $J_{\beta^*}$ admit an  explicit expression analogous to \eqref{eq:sum_of_p_norm} and  \eqref{eq:Jalpha_explicit}, precisely
$$
\|x\|_{\beta^*} = \Big(\sum_{k=1}^h \norm{x}_{\varpi_k^*, q_k^*}^{t^*}\Big)^{1/t^*} \quad \text{and} \quad 
J_{\beta^*}(x) = \|x\|_{\beta^*}^{1-t^*}\sum_{k=1}^h\norm{x}_{\varpi_k^*,q_k^*}^{t^*-q_k^*} \mathrm{diag}(\varpi_k^*)\Phi_{q_k^*}(x),
$$
for all choices of the weights $\varpi_i$ such that $\varpi_1+\dots+\varpi_h>0$. %and where $t^* = t/(t-1)$ and $q_k^* = q_k/(q_k-1)$ for all $k$.

Thus, we obtain an explicit formula for the operator %$\mathcal S_A$ 
$$
\SS_A(x)=J_{\beta^*}(A^\top J_{\alpha}(Ax)) 
$$
which allows us to easily implement the power method \eqref{eq:power_sequence} for the matrix norm  $\|A\|_{\beta\to\alpha}$. Moreover, if we let $\mathrm{nnz}(X)$ denote the number of nonzero entries in $X$ and  we  assume arithmetic operations have unit cost, this also implies that  evaluating $J_{\alpha}$ and $J_{\beta^*}$ costs $\mathcal O(\sum_{i=1}^d\mathrm{nnz}(\omega_i))$ and $\mathcal O(\sum_{i=1}^h\mathrm{nnz}(\varpi_i))$ operations, respectively. % $\mathcal{O}(m)$ and $\mathcal{O}(n)$ operations respectively. 
So, the total cost of evaluating $\SS_A$ (i.e.\ of each iteration of the method) is $\mathcal O\big(C(\SS_A)\big)$ where %given by
$$
C(\SS_A) = \sum_{i=1}^d\mathrm{nnz}(\omega_i) + \sum_{i=1}^h\mathrm{nnz}(\varpi_i) + \mathrm{nnz}(A)
$$
which boils down to $\mathcal O(dn+hn+n^2)$ when all the $\omega_i$, $\varpi_i$ and $A$ are full. 
%
%is dominated by the matrix--vector multiplication with $A$ and $A^T$ and thus requires $\mathcal O(\mathrm{nnz}(A))$ operations, where $\mathrm{nnz}(A)$ denotes the number of nonzero entries of $A$.

As a consequence, we have
\begin{theorem}\label{sumPF}
Let $A\in\RR^{m\times n}$ be a nonnegative matrix such that $A^\top A$ is irreducible. Let  $\norm{\cdot}_{\alpha}$ and $\norm{\cdot}_{\beta}$ be as in \eqref{eq:sum_of_p_norm} and \eqref{eq:beta_norm_special}, respectively. Let
$$\tau=\kappa_H(A)\kappa_H(A^\top)\Big(s-1+\sum_{k=1}^d\max\{0,p_k-s\}\Big)\Big(t-1+\sum_{j=1}^h\max\{0,q_j-t\}\Big).$$
If $\tau<1$ and $\norm{\cdot}_{\alpha}$, $\norm{\cdot}_{\beta}$ are differentiable, then $\norm{A}_{\beta^*\to\alpha}$ can be approximated to $\epsilon$ precision in 
$\mathcal{O}\big(C(\SS_A)\ln(1/\epsilon)\big)$ arithmetic operations with the power sequence~\eqref{eq:power_sequence}.
\end{theorem}
\begin{proof}%[Proof of Theorem \ref{sumPF}]
% Lemma \ref{lem:upperbound_pH} provides the upper bounds for $\kappa_H(J_\alpha)$ and $\kappa_H(J_\beta)$. The conclusion is obtained in the same way as the proof of Theorem \ref{pqPF}.

% \begin{proof}
Besides the complexity bound, the result is a direct consequence of Theorem \ref{newPF} and the upper bounds for $\kappa_H(J_\alpha)$ and $\kappa_H(J_\beta)$ obtained in Lemma \ref{lem:upperbound_pH}. Let us provide and estimates for the total number of operations required by the fixed point sequence \eqref{eq:power_sequence}.
% As $J_\alpha(x)=\norm{x}_q^{1-q}\Phi_q(x)$ and $J_{\beta^*}(y)=\norm{x}_p^{1-p}\Phi_p(x)$, $J_p$ and $J_q$ can be respectively evaluated in $\mathcal{O}(m)$ and $\mathcal{O}(n)$ arithmetic operations. Hence, $\SS_A$ can be evaluated in $\mathcal{O}(m \, n)$ operations. 
Let $\tilde C$ be as in Theorem \ref{newPF}. We have $\tilde C\tau^k <\epsilon$ if and only if $k>(\ln(\epsilon)-\ln(\tilde C))/\ln(\tau)$. 
As $(\ln(\epsilon)-\ln(\tilde C))/\ln(\tau)\in\mathcal{O}(-\ln(\epsilon))$ for $\epsilon\to 0$, we deduce that $\norm{A}_{q\to p}-\epsilon\leq\norm{Ax_k}_{p}$ after $\mathcal{O}(\ln(\epsilon^{-1}))$ iterations of $\SS_{A}$, leading to a total complexity of $\mathcal{O}(C(\mathcal S_A)\ln(\epsilon^{-1}))$.
% \end{proof}
\end{proof}
We conclude the section by proving a number of corollaries of Theorem \ref{sumPF} that illustrate the richness of the class of problem that can be addressed via that theorem. For simplicity, in the statements we assume that the involved matrices are square and positive. However, more general statements involving irreducible and rectangular matrices can be easily derived by reproducing the proof of the corresponding corollary.
\begin{corollary}\label{corrr1}
Let $A\in\RR^{n\times n}$ be a positive matrix. Let $\omega,\varpi\in\RR^n$ be positive weights and $1<p,q<\infty$. Let
$$\norm{A}_{\beta\to\alpha}=\max_{x\neq 0}\frac{\norm{Ax}_{\omega,p}}{\norm{x}_{\varpi,q}}\qquad \qquad \text{and}\qquad\qquad \tau = \kappa_H(A)^2\frac{p-1}{q-1}.$$ 
It $\tau<1$, then $\norm{A}_{\beta\to\alpha}$ can be computed to $\epsilon$ precision in $\mathcal{O}\big(\mathrm{nnz}(A)\ln(1/\epsilon)\big)$ operations.
\end{corollary}
\begin{proof}
As $d=h=1$, $C(\SS_A)=\mathrm{nnz}(A)$,  $\norm{y}_{\alpha}=\norm{y}_{\omega,p}$ and $\norm{x}_{\beta^*}=\norm{x}_{\varpi^*,q^*}$ in Theorem \ref{sumPF}.
\end{proof}
\begin{corollary}
Let $A,B\in\RR^{n\times n}$ be positive matrices. Further, let $1<p,q,r<\infty$,
$$ \left\| \bigg[\begin{matrix}A\\ B\end{matrix}\bigg] \right\|_{\beta\to\alpha}\!\!\!\!=\max_{x\neq 0}\frac{2\,\norm{Ax}_{p}+3\,\norm{Bx}_{q}}{\norm{x}_{r}} \qquad \text{and}\qquad \tau = \kappa_H\Big(\bigg[\begin{matrix}A\\ B\end{matrix}\bigg]\Big)^2\,\frac{p+q-2}{r-1}.$$ 
If $\tau<1$, then $\left\| \bigg[\begin{matrix}A\\ B\end{matrix}\bigg] \right\|_{\beta\to\alpha}$ \!\!\!\!\! can be computed to $\epsilon$ precision in $\mathcal{O}\big(N\,\ln(1/\epsilon)\big)$ operations with $N=\mathrm{nnz}(A)+\mathrm{nnz}(B)$.
\end{corollary}
\begin{proof}
Let $d=2,h=1$, $\omega_i=2,\varpi_i=3$ for $i=1,\ldots,n$, and  $\norm{x}_{\beta^*}=\norm{x}_{r^*}$, $\norm{(y,z)}_{\alpha}=\norm{\norm{y}_{\omega,p},\norm{z}_{\varpi,q}}_1$ in Theorem \ref{sumPF}. Also note that $\mathcal O(C(\SS_A))=\mathcal O(N)$.
\end{proof}
\begin{corollary}
Let $A\in\RR^{n\times n}$ positive, $1<p<\infty$, $2\leq q,r<\infty$, 
$$\norm{A}_{\beta\to\alpha}=\max_{x+y\neq 0}\frac{\norm{Ax+Ay}_{p}}{\sqrt{\norm{x}_q^2+\norm{y}^2_r}} \qquad \text{and}\qquad \tau = \kappa_H(A)^2(p-1).$$ 
If $\tau<1$, then $\norm{A}_{\beta\to\alpha}$ can be computed to $\epsilon$ precision in $\mathcal{O}\big(\mathrm{nnz}(A)\ln(1/\epsilon)\big)$ operations.
\end{corollary}
\begin{proof}
Let $d=1,h=2$, $\norm{y}_{\alpha}=\norm{y}_p$, $\norm{x}_{\beta^*}=\norm{\norm{x}_{q^*},\norm{x}_{r^*}}_2$ in Theorem \ref{sumPF}.
\end{proof}
\begin{corollary}
	Let $A,B\in\RR^{n\times n}$ be positive matrices,  $1<s\leq\theta\leq p,q,r<\infty$,
	$$\norm{[A\ B]}_{\beta\to\alpha}^\theta=\max_{(x,y)\ne (0,0)}\frac{\norm{Ax}^\theta_p+\norm{By}_{q}^\theta}{\norm{x}_r^\theta+\norm{y}^\theta_s} \quad \text{and}\quad \tau = \kappa_H([A\ B])^2\,\frac{p+q-\theta-1}{s-1}.$$ 
	If $\tau<1$, then $\norm{[A\ B]}_{\beta\to\alpha}$ can be computed to $\epsilon$ precision in $\mathcal{O}\big(N\,\ln(1/\epsilon)\big)$ operations with $N=\mathrm{nnz}(A)+\mathrm{nnz}(B)$.
\end{corollary}
\begin{proof}
	Let $d=2,h=2$, $\norm{(y,z)}_{\alpha}=\norm{\norm{y}_p,\norm{z}_q}_\theta$ and $\norm{x}_{\beta^*}=\norm{\norm{x}_{r^*},\norm{x}_{s^*}}_{\theta^*}$ in Theorem \ref{sumPF}.
\end{proof}
\begin{corollary}
	Let $A,B\in\RR^{n\times n}$ be positive matrices and $1<p,q,r<\infty$, let
	$$\phi=\max_{(x,y,u+v)\ne (0,0,0)}\min\bigg\{\frac{\norm{A(x+y)+B(u+v)}_{p}}{\norm{(x,u)}_q},\frac{\norm{A(x+y)+B(u+v)}_{p}}{\norm{(y,v)}_r}\bigg\}.$$
	If $\tau = \frac{p-1}{q-1}+\frac{p-1}{r-1}<1$, then $\phi$ can be computed to $\epsilon$ precision in $\mathcal{O}\big(N\,\ln(1/\epsilon)\big)$ operations with $N=\mathrm{nnz}(A)+\mathrm{nnz}(B)$.
\end{corollary}
\begin{proof}
	Let $M=\Big[\begin{matrix}A& A &0 \\ 0 & B & B \end{matrix}\Big]$, $d=2,h=2$, $\norm{y}_{\alpha}=\norm{y}_{p}$ and $\norm{(x,y,z)}_{\beta^*}=\norm{\norm{(x,z)}_{r^*},\norm{(y,z)}_{s^*}}_{1}$ in Theorem \ref{sumPF}. Note that $\kappa_H(M)=1$ by Theorem \ref{computeBirkhoff}.
\end{proof}
\begin{corollary}\label{corrrend}
    Let $A,B\in\RR^{n\times n}$ be positive matrices and $1<p,q,r<\infty$. Let $\sigma_p\colon\RR^n\to\RR^n_+$ be defined as $\sigma_p(x)=(|x_1|^p,\ldots,|x_n|^p)^\top$ and let
    $$\norm{B}_{\beta\to \alpha}^{p}=\max_{\norm{x}_r=1}\norm{A\sigma_p(Bx)}_{q}\qquad \text\qquad \tau = \frac{pq-1}{r-1}\kappa(B)\kappa(B^\top).$$
    If $\tau<1$ then $\norm{B}_{\beta\to \alpha}$ can be computed to $\varepsilon$ precision in $\mathcal{O}\big(N\,\ln(1/\epsilon)\big)$ operations with $N=\mathrm{nnz}(A)+\mathrm{nnz}(B)$.
\end{corollary}
\begin{proof}
As $x\mapsto A\sigma_p(Bx)$ is positively homogeneous of degree $p$, we have
$$\max_{\norm{x}_r=1}\norm{A\sigma_p(Bx)}_q=\max_{x\neq 0}\frac{\norm{A\sigma_p(Bx)}_q}{\norm{x}_r^p}=\Big(\max_{x\neq 0}\frac{\norm{A\sigma_p(Bx)}_q^{1/p}}{\norm{x}_r}\Big)^p.$$
Let $\norm{\cdot}_{\new{\beta^*}}=\norm{\cdot}_{r^*}$ and $\norm{x}_{\alpha}=\norm{A\sigma_p(x)}_q^{1/p}$. %for all $x\in\RR^n$. 
Then, $\norm{Bx}_{\alpha}=\norm{A\sigma_{p}(Bx)}_q^{1/p}$ and with $\omega_i = (A_{i,1},\ldots,A_{i,n})$, it holds $\norm{x}_{\alpha}=\norm{(\norm{x}_{\omega_1,p},\ldots,\norm{x}_{\omega_n,p})}_{pq}$ for every $x$. The proof is now a direct consequence of Theorem \ref{computeBirkhoff} with $s=n$ and $t=1$.
\end{proof}

\newcommand{\D}{\mathcal{D}}
\newcommand{\E}{\mathcal{E}}
\newcommand{\V}{\mathcal{V}}
\newcommand{\pps}[1]{\langle #1 \rangle_{\pi}}
\section{Application to the estimation of the log-Sobolev constant of Markov chains}\label{sec:log-sobolev}

In this final section we discuss an intriguing relation between our Theorem \ref{newPF} and the logarithmic Sobolev constant of Markov chains. This constant is widely studied and is particularly useful in proving convergence estimates of Markov chain Monte Carlo algorithms (see e.g. \cite{carbone,diaconis1996,goel,jerrum}). While upper bounds for this constant are relatively simple to obtain, lower bounding the log--Sobolev constant is a challenging task  \cite{lacoste}. By exploiting the hypercontractive inequalities that characterize the log--Sobolev constant in terms of suitable weighted matrix norms \cite{bakry,gross} we prove a new lower bound given in terms of the Birkhoff--Hopf contraction ratio of the continuous time Markov semigroup associated to the chain. In particular, Theorem \ref{newPF} plays a critical role in our derivation as it allows us to compute the norm $\|A\|_{2\to q}$ with $q>2$, which is precisely the type of norms that appear in the aforementioned hypercontractive inequalities.
\subsection{The log-Sobolev constant and hypercontractive inequalities}
We start by recalling the log-Sobolev constant   and the corresponding hypercontractive inequalities. 

Let $(K,\pi)$ be a finite Markov chain with positive stationary distribution $\pi=(\pi_1,\ldots,\pi_n)^\top\in\RR^n_{++}$, i.e. $K\in\RR^{n\times n}$ is a nonnegative matrix such that $K\ones = \ones$, $\pi^\top = \pi^\top K$ and $\norm{\pi}_{1}=1$, 
where $\ones\in\RR^n_{++}$ denotes the vector of all ones. We say that $(K,\pi)$ is irreducible if $K$ is irreducible and, in this case, $\pi$ is automatically a positive 
probability vector. Now, consider the diagonal matrix $D_{\pi}=\diag(\pi)$ and the weighted inner product $\pps{\cdot,\cdot}\colon\RR^n\times \RR^n\to \RR$ defined as 
$\pps{x,y}=\ps{D_{\pi}x}{y}$. For a matrix $M\in\RR^{n\times n}$ let us denote by $M^*$ the adjoint of $M$ with respect to $\pps{\cdot,\cdot}$, i.e. $M^*=D_{\pi}^{-1}M^T D_{\pi}$. Furthermore for $p,q\in(1,\infty)$, let $\norm{\cdot}_{\pi,p}$ and $\norm{\cdot}_{\pi,p\to q}$ be the weighted $\ell^p$-norm and weighted matrix $\ell^{p,q}$-norm defined for every $x\in\RR^n$ and $M\in\RR^{n\times n}$ as:
$$\norm{x}_{\pi,p} = \Big(\sum_{i=1}^n\pi_i|x_i|^p\Big)^{1/p}, \qquad \norm{M}_{\pi,p\to q}=\max_{x\neq 0}\frac{\norm{Mx}_{\pi,q}}{\norm{x}_{\pi,p}}.$$

A Sobolev inequality is an inequality relating the Dirichlet form and the entropy induced by $(K,\pi)$. These two quantities are respectively defined as
\begin{equation}\label{defdiri}
	\D(x,y)=\pps{x, (I-K)y} \qquad \forall x,y\in\RR^n 
\end{equation}
and 
$$\E(x)=\sum_{i=1}^n |x_i|^2\log\Big(\frac{|x_i|^2}{\norm{x}_{\pi,2}^2}\Big)\pi_i\qquad \forall x\in\RR^n,x\neq 0.$$
The log-Sobolev constant $\sigma$ of $(K,\pi)$ is then defined as
\begin{equation}\label{logsob2}
\sigma= \sup\big\{s\geq 0 \ \big|\ s\,\E(x) \leq \D(x,x)\quad \forall x\in\RR^n_{++}\big\}.
\end{equation}
In particular, note that when $x=y$ in \eqref{defdiri} we have 
$$
	\D(x,x)=\pps{x,(I-\tfrac{1}{2}(K+K^*))x}\, ,
$$
from which it follows that $\sigma\leq \lambda/2$ (see e.g.\ \cite[\S 2]{goel}), 
%It is known (\cite[\S 2]{goel}) that 
where $\lambda$ is the spectral gap of $(K,\pi)$, i.e.\ the smallest non-zero eigenvalue of $I-\tfrac{1}{2}(K+K^*)$.
Furthermore note that the log-Sobolev constants of $(K,\pi)$ and $(\tfrac{1}{2}(K+K^*),\pi)$ coincide and $(\tfrac{1}{2}(K+K^*),\pi)$ is reversible \cite{diaconis1996}.

The continuous time Markov semigroup $\{H_t\}_{t>0}$ induced by $(K,\pi)$ is defined as 
\begin{equation}
\label{defHt}H_t = \exp\big(-t(I-K)\big)=e^{-t}\sum_{j=0}^\infty \frac{t^j}{j!} K^{j}\qquad \forall t>0.
\end{equation}
Note that $H_t$ is always a stochastic %nonnegative 
matrix and it is positive if $K$ is irreducible.

In order to avoid possible confusion, in the following we denote the exponential of a matrix $M\in\RR^{n\times n}$ by $\exp(M)$ and the exponential of a number $x\in\RR$ by $e^x$. 

The following theorem characterizes the log-Sobolev constant $\sigma$  in terms of the operator norm $\norm{H_t}_{\pi,2\to q}$.
\begin{theorem}[Theorem 3.5, \cite{diaconis1996}]\label{hyperineq} 
	Let $(K,\pi)$ be a finite Markov chain with log-Sobolev constant $\sigma$. Then
	\begin{enumerate}%[\quad i)]
	\item Assume that there exists $\beta>0$ such that $\norm{H_t}_{\pi,2\to q}\leq 1$ for all $t>0$ and $2\leq q <\infty$ satisfying $e^{4\beta t}\geq q-1$, then $\beta \leq \sigma$.\label{e1diaco}
	\item Assume that $(K,\pi)$ is reversible. Then $\norm{H_t}_{\pi,2\to q}\leq 1$ for all $t>0$ and $2\leq q <\infty$ satisfying $e^{4\sigma t}\geq q-1$.
    %\item 	\red{\sout{For nonreversible chains, we still have $\norm{H_t}_{\pi,2\to q}\leq 1$ for all $t>0$ and $2\leq q <\infty$ satisfying $e^{2\sigma t}\geq q-1$.}}
	\end{enumerate}
\end{theorem}
A more general version of Theorem \ref{hyperineq} can be found in \cite{gross} where a characterization of $\sigma$ is given in terms of $\norm{H_t}_{\pi,p\to q}$. 
\subsection{Lower bounds via the Birkhoff contraction rate}
Let $(K,\pi)$ be a Markov chain such that $\pi$ is positive and $K+K^*\in\RR^{n\times n}$ is irreducible. 
The log-Sobolev constant $\sigma$ and the Birkhoff contraction rate $\kappa_H(H_t)$, where $\{H_t\}_{t>0}$ is the continuous semi-group \eqref{defHt}, can be directly connected by using properties of Markov chains and our new Theorem \ref{newPF}. Before discussing our main result, we prove the following theorem whose proof illustrates the mechanism behind the connection between $\sigma$ and $\kappa_H(H_t)$.
%Note that as $K\ones=\ones$ and $\pi^\top K=\pi$, by construction, we have $H_t\ones = \ones$ and $\pi^\top H_t=\pi$ as well. This simple observation allows to show that $\norm{H_t}_{\pi,2\to q}=1$ whenever the assumptions of Theorem \ref{newPF} are fulfilled. Precisely, we prove the following:
\begin{theorem}\label{corolexp}
	Let $M\in\RR^{n\times n}_+$ be a stochastic matrix and let $\pi\in\RR^n$ be a probability vector satisfying $\pi^TM=\pi^T$. If $MM^*$ is irreducible, then it holds
	$\norm{M}_{\pi,2\to q}=1$ for every $1\leq q\leq 1+\kappa_H(MM^*)^{-1}.$
\end{theorem}
\begin{proof}
We note that $\|\cdot\|_{\pi,p^*}$ is the dual norm of $\|\cdot\|_{\pi,p}$ for $\frac{1}{p}+\frac{1}{p^*}=1$ with respect to $\ps{\cdot}{D_\pi \cdot}$. 
It follows that 
\begin{align*}
\max_{x\neq 0} \frac{\|Ax\|_{\pi,p}}{\norm{x}_{\pi,r}}=\max_{x,y\neq 0} \frac{\ps{y}{D_\pi Ax}}{\norm{y}_{\pi,p^*}\norm{x}_{\pi,r}}
= \max_{x,y\neq 0}\frac{\ps{A^*y}{D_\pi x}}{\norm{y}_{\pi,p^*}\norm{x}_{\pi,r}} = \max_{x\neq 0} \frac{\|A^*x\|_{\pi,r^*}}{\norm{x}_{\pi,p^*}},
\end{align*}
where $A^*$ is the adjoint of $A$ given as $A^*=D_{\pi}^{-1}A^T D_{\pi}$. This shows that
$\|A\|_{\pi,r\to p}=\|A^*\|_{\pi,p^* \to r^*}$. The iterator for $\|M^*\|_{\pi,q^* \to 2}$ is
\begin{equation}\label{SS2p} \SS_{M^*}(x) = \|MM^*x\|_q^{1-q} \Phi_q(MM^*x).
\end{equation}
and thus $\kappa_H(\SS_{M^*})\leq (q-1)\kappa_H(MM^*)$. It holds $M^* \ones=\ones$ and $M\ones=\ones$ and thus $\SS_{M^*}(\ones)=\ones$. If 
 $(q-1)\kappa_H(MM^*)<1$ then, by Theorem \ref{newPF}, $\SS_{M^*}$ has a unique fixed point (up to scaling) and thus $\|M\|_{\pi,2 \to q} = \frac{\|M^* \ones\|_{\pi,2}}{\|\ones\|_{\pi,q^*}}
= 1$ for all $1<q<1 + \kappa_H(MM^*)^{-1}$. Finally, by the continuity of $q \mapsto \|M\|_{\pi,2 \to q}$ it holds $\|M\|_{\pi,2 \to 1}=\|M\|_{\pi,2 \to 1+\kappa_H(MM^*)}=1$.
\end{proof}

Few relevant observations on $\kappa_H(MM^*)$ are in order. In the proof of Theorem \ref{corolexp}, we show that $\norm{M}_{\pi,2\to q}=\norm{M^*}_{\pi,q^*\to 2}$ and $\norm{M^*}_{\pi,q^*\to 2}$ can be computed by finding the fixed points of $\SS_{M^*}$ defined in \eqref{SS2p}. Note that $\SS_{M^*}$ has a simpler form than $\SS_M$ because it is the composition of just one nonlinear and one linear mapping. This is because the Euclidean norm $\|\cdot\|_2$ is used in the numerator of $f_{M^*}$. This simplification is interesting as it implies that $\kappa_H(\SS_{M^*})=(q-1)\kappa_H(MM^*)$ because $J_q$ is a dilatation and the Birkhoff-Hopf theorem gives the best Lipschitz constant. Hence, the proof of Lemma 6.2 is tight in the sense that it uses Theorem \ref{newPF} with the best possible estimate $\tau$ on $\kappa_H(\SS_{M^*})$.

Now, Theorem \ref{corolexp} implies that $\norm{H_t}_{\pi,2\to q}=1$ for every $1\leq q \leq 1 + \kappa_H(H_tH_t^*)^{-1}$. The hyper-contractive inequalities of Theorem \ref{hyperineq} make now clear that $\sigma$ and $t\mapsto \kappa_H(H_tH_t^*)$ are related. We describe this relation in the following final Theorem \ref{mainsobo}, whose proof combines the following further preliminary lemma  with a number of properties of subadditive functions, i.e.\ functions $g$ satisfying $g(s+t)\leq g(s)+g(t)$ for all $s,t$. 
%
%
%
%
%
%We first need one further  lemma 
% which combines Theorem \ref{corolexp} with Theorem 3.2 in \cite{bakry}. %, to obtain an improved version of \eqref{e1diaco} in Theorem \eqref{hyperineq}.
\begin{lemma}\label{PFSobo}
Let $(K,\pi)$ be a Markov chain such that $\pi\in\RR^n$ is positive and $K+K^*\in\RR^{n\times n}$ is irreducible. Let $T>0$ and $f\colon[0,T)\to (0,\infty)$ be continuous decreasing and right-differentiable at $0$. 
If $f(t)\leq \kappa_H(\exp(\tfrac{t}{2}(K+K^*))$ for every $t\in[0,T),$ then $-\frac{f'(0)}{f(0)}\leq 2\,\sigma,$ where $\sigma$ is the log-Sobolev constant of $(K,\pi)$. 
\end{lemma}
\begin{proof}
The proof combines Theorem \ref{corolexp} with Theorem 3.2 in \cite{bakry}. 	If $f'(0)=0$ the result is trivial, so assume $f'(0)<0$.
	The identity \eqref{defdiri} implies that the log-Sobolev constant of $(K,\pi)$ equals that of $(\tfrac{1}{2}(K+K^*),\pi)$.
	Let 
	\begin{equation}\label{defMt}M_t = \exp\big(-t\big(I-\tfrac{1}{2}(K+K^*)\big)\big) \qquad \forall t>0.
	\end{equation}
	Then, $\{M_t\}_{t>0}$ is the continuous Markov semi-group of $(\tfrac{1}{2}(K+K^*),\pi)$. 
	The equality $e^{-t}\exp(\tfrac{t}{2}(K+K^*))=M_t$ implies $\kappa_H(M_t)=\kappa_H(\exp(\tfrac{t}{2}(K+K^*))$ for all $t\geq 0$.
	As $K+K^*$ is reversible, we have $M_t=M_t^*$ and thus, by Theorem \ref{corolexp}, it holds
	$\norm{M_t}_{\pi,2\to q}=1$ for all $1\leq q \leq 1+\kappa_H(M_t^2)^{-1}$.
	As $M_t^2 = M_{2t}$, we have 
	$\kappa_H(M_t^2)^{-1}\leq f(2t)^{-1}$ for all $t\in [0,T/2)$. Set $q(t)=1+f(0)/f(2t)$, then $q\colon [0,T/2)\to [2,\infty)$ is increasing, continuous, right differentiable at $0$ and by Theorem \ref{corolexp} it holds $\norm{M_t}_{\pi,2\to q(t)}=1$ for all $t\in [0,T/2)$. Furthermore, $q(0)=2$ and so by Theorem 3.2 in \cite{bakry}, we have
	$q'(0)\leq 4\,\sigma$. As $q'(0)=-2\, f'(0)/f(0)$, this concludes the proof.
\end{proof}
% We prove Theorem \ref{mainsobo} by combining the above results together with a number of properties of subadditive functions, i.e.\ functions $g$ satisfying $g(s+t)\leq g(s)+g(t)$ for all $s,t$.
\begin{theorem}\label{mainsobo}
Let $(K,\pi)$ be a Markov chain such that $\pi\in\RR^n$ is positive and $K+K^*\in\RR^{n\times n}$ is irreducible. %Denote by $\kappa_H(B)$ the Birkhoff contraction rate of a matrix $B$ and 
For $t\geq 0$, let 
$\rho(t)\,=\,\kappa_H(\exp(\tfrac{t}{2}(K+K^*)) $. 
Then, it holds
$$\lim_{t\to 0}\,\rho(t)^{-1/t} \leq\, e^{2 \, \sigma}\qquad$$ 
where $\sigma$ is the log-Sobolev constant of $(K,\pi)$. 
\end{theorem}
\begin{proof}[Proof of Theorem \ref{mainsobo}]
%Our proof uses a result of subadditive functions. 
Let $M_t$ be as in \eqref{defMt} so that $\kappa_H(M_t)=\rho(t)$ for all $t\geq 0$. Define $g\colon (0,\infty)\to (-\infty,0)$ as $g(t)=\ln(\kappa_H(M_t))$. Recalling that $M_{t+s}=M_tM_s$ and $\kappa_H(AB)\leq \kappa_H(A)\kappa_H(B)<1$ for positive $A,B$, one deduce that $g(t)$ is subadditive. In particular, as $\lim_{t\to 0}g(t)=0$, Theorems 7.6.1 and 7.11.1 in \cite{hille1996functional} imply that there exists $\gamma\in (-\infty,0]$ such that % and $\xi\in [-\infty,0)$ such that
$$\gamma=\sup_{t>0}\frac{g(t)}{t}=\lim_{t\to 0}\frac{g(t)}{t}\, .$$ %\quad \text{and}\quad \xi=\inf_{t>0}\frac{g(t)}{t}=\lim_{t\to \infty}\frac{g(t)}{t}.$$
%To conclude the proof, we show that $-\xi\leq 2\,\sigma$ which, composed by the exponential, implies the statement as $\sup_{t>0}-g(t)/t=-\inf_{t>0}g(t)/t$.

Now, let $T>0$ and $C(T)=\inf_{t\in(0,T)}g(t)/t$. As $\gamma$ exists, $t\mapsto g(t)/t$ can be extended by continuity on $[0,T]$. Hence, $C(T)$ is bounded. By construction, we have $g(s)\geq s\,C(T)$ for all $s\in(0,T)$. Composing by the exponential function, we obtain $\kappa_H(M_s)\geq f(s)=e^{s\,C(T)}$ for all $s\in [0,T)$. Hence, by Lemma \ref{PFSobo}, we deduce that $0\leq -C(T)\leq 2 \sigma$. By construction we have $\gamma = \lim_{T\to 0} C(T)$ and thus 
$$
2\sigma \geq -\gamma = -\lim_{t\to 0}\frac{\ln \rho(t)}{t} = \ln\left(\lim_{t\to 0}\rho(t)^{-1/t}\right)$$
which concludes the proof.
\end{proof}
We conclude with a corollary that explicitly shows the bound ensured by Theorem \ref{mainsobo} on a general Markov chain on a two-state space.
\begin{corollary}
Let $a,b\in (0,1]$ and consider 
$$K = \begin{bmatrix} 1-a & a\\ b & 1-b\end{bmatrix}\qquad \text{and}\qquad \pi=\frac{1}{a+b}\begin{bmatrix}b\\ a \end{bmatrix}.$$
Then $(K,\pi)$ is an irreducible Markov chain and 
\begin{equation}\sqrt{ab}  \leq  \sigma=\begin{cases}\frac{a-b}{\ln(a)-ln(b)} & a\neq b\\ a & \text{otherwise}\end{cases}
\end{equation}
where $\sigma$ is the log-Sobolev constant of $(K,\pi)$.
\end{corollary}
\begin{proof}
The formula for $\sigma$ is proved in Theorem 2.2 of \cite{chen2008logarithmic}. 
As $K^*=K$ we have $\rho(t)=\kappa_H(\exp(\frac{t}{2}(K+K^*)))=\kappa_H(\exp(t K))$. A direct computation shows that, with $\xi = a+b$ and $c=\sqrt{a/b}$, it holds
		$$\exp(t\, K) =\frac{\sqrt{a\,b}\,e^{t}}{\xi}\begin{bmatrix}1+c^{2}\,e^{-\xi\,t}& 1-e^{-\xi\,t}\\ 1-e^{-\xi\,t}&1+c^{-2}\,e^{-\xi\,t} \end{bmatrix}\begin{bmatrix}c^{-1}&0\\0&c\end{bmatrix}\qquad \forall t>0.$$
With the help of the formula of Theorem \ref{computeBirkhoff}, we obtain 
		$$\rho(t)=\frac{\sqrt{(1+c^2\,e^{-\xi\,t})(1+c^{-2}\,e^{-\xi\,t})}+e^{-\xi\, t}-1}{\sqrt{(1+c^{2}\,e^{-\xi\,t})(1+c^{-2}\,e^{-\xi\,t})}-e^{-\xi\, t}+1}\qquad \forall t>0$$
and thus, as $\xi=a+b$ we have
		$$ \lim_{t\to 0}\rho(t)^{-1/t}=e^{2\sqrt{ab}}. $$
		which, together with Theorem \ref{mainsobo}, concludes the proof.  
\end{proof}

\section{Conclusions}
On top of being a classical problem in numerical analysis, computing the norm of a matrix $\|A\|_{\beta\to\alpha}$ is a problem that appears in many recent applications in data mining and optimization. However, except for a few choices of $\|\cdot\|_\alpha$ and  $\|\cdot\|_\beta$, computing such a matrix norm to an arbitrary precision is generally unfeasible for large matrices as this is known to be an NP-hard problem.
The situation is different when the matrix has nonnegative entries, in which case $\|A\|_{q\to p}$ is known to be computable for $\ell^p$ norms such that $q\leq p$. In this paper we have both (a) refined this result, by showing that the condition $p<q$ is not necessarily required and (b) extended this  result to much more general vector norms $\|\cdot\|_\alpha$ and  $\|\cdot\|_\beta$ than $\ell^p$ norms. In particular, we have shown how to compute matrix norms induced by monotonic norms of the form $\norm{x}_{\alpha}=\norm{\big(\norm{x}_{\alpha_1},\ldots,\norm{x}_{\alpha_d}\big)}_{\gamma}$, where  we also allow $\norm{x}_{\alpha_i}$ to measure only a subset of the coordinates of $x$. Using these kinds of norms we can globally solve in polynomial time quite sophisticated nonconvex optimization problems, as we discuss in the examples corollaries at the end of Section \ref{sec:sum_p_norms}. 
Moreover, we emphasize that  our result shows for the first time that the norm $\|A\|_{2\to q}$  with $q>2$ is computable when $A$ has positive entries. This kind of norms appear frequently in hypercontractive inequalities and as a nontrivial application of this result we eventually provide a new lower bound for the logarithmic Sobolev  constant of a Markov chain.

% \bibliographystyle{plain}
% \bibliography{references}

\end{document}